%% file: main.tex
\documentclass[final,3p,times]{elsarticle}

\usepackage{amssymb, amsmath, amsfonts}
\usepackage{booktabs}
\usepackage{amsthm}
\usepackage{tikz}
\usepackage{pgfplots}
\renewcommand\arraystretch{1.5}
\pgfplotsset{compat=newest}
\usetikzlibrary{shapes.geometric,decorations.pathmorphing,patterns, shapes.symbols, arrows.meta}
\usetikzlibrary{calc,angles}
\tikzset{block/.style={draw,rectangle, align=center, inner sep=5pt},
oval/.style={draw,ellipse, align = center, inner sep=5pt},
decision/.style={draw,diamond, aspect = 2, align=center,inner sep=2pt}}

\usepackage{paralist}
\usepackage{natbib}
\usepackage{hyperref}
\hypersetup{
    colorlinks=true,
    linkcolor=blue,
    filecolor=magenta,      
    urlcolor=cyan,
}
\usepackage[noabbrev]{cleveref}
\usepackage[font=small,labelfont=bf]{caption}
\usepackage{graphicx}
\graphicspath{{./figures/}}
\usepackage{comment}

\makeatletter
\def\namedlabel#1#2{\begingroup
   \def\@currentlabel{#2}%
   \label{#1}\endgroup
}
\makeatother

\journal{}

\input{lvc_style}

\input{thm_style}
\usepackage{upgreek}
\biboptions{sort&compress}
\allowdisplaybreaks
\usepackage{graphicx}
\usepackage{caption}
\usepackage{subcaption}
\usepackage{multirow}
\usepackage{siunitx}
\usepackage{adjustbox}
\usepackage{enumitem}

\crefname{assumption}{Assumption}{Assumptions}
\Crefname{assumption}{Assumption}{Assumptions}
\crefname{enumi}{Assumption}{Assumptions}
\Crefname{enumi}{Assumption}{Assumptions}

\begin{document}

\begin{frontmatter}



\title{Operator preconditioning of the Costabel FEM-BEM coupling for the time-harmonic Maxwell transmission problem}


\author[1,2]{Quang Huy Nguyen\corref{cor}}
\ead{quanghuy.nguyen@ugent.be}

\author[1]{Van Chien Le}
\ead{vanchien.le@ugent.be}

\author[1]{Kristof Cools}
\ead{kristof.cools@ugent.be}

\address[1]{IDLab, Department of Information Technology, Ghent University - imec, 9000 Ghent, Belgium.} 

\address[2]{Faculty of Mathematics and Informatics, Hanoi University of Science and Technology, 11657 Hanoi, Vietnam.} 

\cortext[cor]{Corresponding author}

\begin{abstract}
    This paper derives an equivalent formulation of the Costabel FEM-BEM coupling for the time-harmonic Maxwell transmission problem with a Lipschitz interface, and constructs a block-diagonal preconditioner for the resulting discrete system. The reformulated coupling is posed on a product of trace spaces and is obtained by introducing a trace lifting operator and the Steklov--Poincar{\' e} operator. At the algebraic level, it corresponds to eliminating the interior FEM unknowns, resulting in a discrete system that involves the Schur complement of the interior FEM block. Assuming that the frequency lies outside the resonance sets, we prove that the reformulated coupling is equivalent to the Costabel coupling at both the continuous and discrete levels. Using the operator preconditioning framework, we build a preconditioner from the Galerkin matrix induced by the Maxwell single layer boundary integral operator. The spectral condition number of the preconditioned matrix is shown to be uniformly bounded. Numerical experiments illustrate the practical effectiveness of the proposed approach.
\end{abstract}




\end{frontmatter}


\section{Introduction}
\label{sec: introduction}

Let $\Om_{\text{s}} \subset \R^3$ be a bounded Lipschitz domain such that $\Om_0 := \R^3\setminus \overline{\Om_{\text{s}}}$ is connected, and denote $\Gm := \partial \Om_{\text{s}}$. The domain $\Om_{\text{s}}$ is occupied by an inhomogeneous isotropic medium with permeability $\mu_{\text{s}}$ and permittivity $\varepsilon_{\text{s}}$. Assume that $\mu_{\text{s}}$ and $\varepsilon_{\text{s}}$ are real-valued, piecewise Lipschitz, and uniformly positive almost everywhere \cite[Sect.~2.1 and 3.1]{HL1996}. The exterior domain $\Om_0$ is filled with a homogeneous isotropic medium with constant permeability $\mu_0>0$ and permittivity $\varepsilon_0>0$. Let $\om >0$ be a fixed angular frequency. Let $\EEs^{\text{inc}}$ be a given incident electric field in $\Om_0$ at frequency $\om$, assumed to satisfy the homogeneous Maxwell equation in an exterior neighborhood of $\Gm$. We seek an electric field $\EEs$, representing the total electric field in $\Om_{\text{s}}$ and the scattered electric field in $\Om_0$, that satisfies the following time-harmonic Maxwell transmission problem:
\begin{equation}
    \label{eq: scattering problem}
    \left\{
    \begin{array}{@{}r@{\;}c@{\;}l@{\qqq}l@{\;}l@{}l@{}}
        \curlt\, \mu_{\text{s}}^{-1}\curlt \EEs - \omega^2 \varepsilon_{\text{s}} \EEs
        & = & \bm{0}
        & \text{in} & \Om_{\text{s}}, & \\
        \curlt\, \mu_0^{-1}\curlt \EEs - \omega^2 \varepsilon_0 \EEs
        & = & \bm{0}
        & \text{in} & \Om_0, & \\
        \gm_D^{\text{s}}\EEs - \gm_D^0 \EEs & = & \gm_D^0 \EEs^{\text{inc}} & \text{on} & \Gm, & \\
        \gm_N^{\text{s}}\EEs - \gm_N^0 \EEs & = & \gm_N^0 \EEs^{\text{inc}} & \text{on} & \Gm, & \\

        \multicolumn{5}{@{}l@{}}{%
        \displaystyle
        \lim_{r\to +\infty}\int_{\partial \mathcal{B}_r}\abs{\curlt \EEs \times\, \mathbf{n}_r + i \kappa\paren{\mathbf{n}_r\times \EEs}\times \mathbf{n}_r}^2\ds = 0.} &
    \end{array}
    \right.
\end{equation}
Here, $\kappa := \om \sqrt{\mu_0\varepsilon_0}$ is the wave number in $\Om_0$. The notations $\gm_D^{\text{s}}, \gm_D^0$ and $\gm_N^{\text{s}}, \gm_N^0$ stand for the Dirichlet and Neumann trace operators mapping fields in $\Om_{\text{s}}$ and $\Om_0$, respectively, to their traces on $\Gm$. Their definitions will be introduced in \Cref{sec: preliminaries}. In the Silver-M{\" u}ller radiation condition, $\mathcal{B}_r$ is the ball of radius $r>0$ centered at the origin, $\mathbf{n}_r$ is the outward unit normal to $\partial\mathcal{B}_r$, and $i$ is the imaginary unit. Problem \eqref{eq: scattering problem} describes the scattering of a time-harmonic electromagnetic wave by an inhomogeneous scatterer. It arises in many applications, including imaging, communication technologies, and biomedical engineering.

Over the past four decades, finite element-boundary element (FEM-BEM) couplings have been studied for a broad class of transmission problems, ranging from elliptic interface problems to acoustic and electromagnetic scattering. In this setting, an inhomogeneous bounded domain is embedded in a homogeneous unbounded medium, which makes the coupling particularly attractive: FEM solves problems in inhomogeneous bounded domains efficiently, while BEM accurately solves problems in the homogeneous unbounded domain. The central idea is to couple the variational formulation for the interior unknown with boundary integral equations for the traces of the exterior unknown through the Steklov--Poincar{\' e} operator. Focusing on the time-harmonic Maxwell transmission problem, several FEM-BEM coupling formulations have been proposed and analyzed \cite{Bourdonnaye1995, ANJ1999, Hiptmair2003, HM2008, BCAG2022}. In \cite{Bourdonnaye1995}, de La Bourdonnaye reduced the coupling formulation into a single boundary equation involving the Steklov–Poincar{\' e} operator and the tangential trace of the electric field. In \cite{ANJ1999}, the authors derived a coupling formulation with a saddle-point structure, in which the unknowns are given by the components arising from the Helmholtz and Hodge decompositions. More recently, Badia et al. \cite{BCAG2022} introduced a weak coupling formulation with optimal transmission operators, which can be solved efficiently by a non-overlapping Schwarz domain decomposition method. These works, however, rely on smooth-interface assumptions. By contrast, the Costabel FEM-BEM coupling \cite{Hiptmair2003, HM2008} appears to be the first formulation for which a rigorous analysis is available with non-smooth, in particular Lipschitz, interfaces. In \cite{Hiptmair2003}, Hiptmair established the well-posedness of the Costabel coupling formulation and of its Galerkin discretization, together with a quasi-optimal a priori error estimate. A stabilized formulation was later developed by Hiptmair and Meury \cite{HM2008} to avoid spurious resonances.

FEM-BEM coupling formulations usually lead to large discrete systems with both sparse blocks stemming from FEM-like discretization in the volume and dense blocks stemming from BEM-like discretization on the boundary, making direct solvers impractical. Consequently, iterative methods appear to be the only viable option. To accelerate the convergence of these iterative methods, different preconditioning techniques have been developed \cite{DZH+2015, JJ2002, GPS2017, GYS2018}. In \cite{DZH+2015}, the Maxwell single layer boundary integral operator was invoked to regularize the Steklov--Poincar{\' e} operator. The authors then proposed and analyzed a Calder{\' o}n preconditioner, relying on the fact that, on smooth surfaces, the regularized Steklov--Poincar{\' e} operator is a Fredholm operator of the second kind. This property, unfortunately, does not extend to the setting of non-smooth surfaces. Other preconditioners have also been proposed for different FEM-BEM coupling formulations in computational electromagnetics \cite{JJ2002, GPS2017, GYS2018}. While these preconditioners have proved effective in many complex engineering applications, the corresponding theoretical results remain limited, possibly because it is challenging to establish a theoretical analysis for these coupling formulations as rigorously and comprehensively as for the Costabel coupling. In a closely related direction, several preconditioners were studied in detail for FEM-BEM coupling formulations for the Laplace and Yukawa interface problems \cite{HS1998, FS2009, FFPS2017, ST2021}. There, the positive definiteness or positive semi-definiteness of the FEM matrices plays a central role in establishing uniform bounds on the condition numbers of the preconditioned matrices. Such arguments are no longer available for the electromagnetic scattering problem considered here.

In this work, we derive an equivalent formulation of the original Costabel coupling \cite{Hiptmair2003} and then construct a preconditioner for the corresponding discrete system. The reformulation is motivated by an idea underlying the Calder{\' o}n preconditioner \cite{DZH+2015}, namely, to use the Maxwell single layer boundary integral operator to regularize the Steklov--Poincar{\' e} operator. To this end, we introduce a trace lifting operator and the Steklov--Poincar{\' e} operator, thereby deriving an equivalent variational formulation on a product of trace spaces. Reformulations using the Steklov--Poincar{\' e} operator, as well as preconditioners for the corresponding discrete systems, have been studied previously \cite{XZ1997, HSW2000, Hu2004}. However, unlike these studies, where the main ingredient is either the positive definiteness or the positive semi-definiteness of the Steklov--Poincar{\' e} operator, here we consider the Steklov--Poincar{\'e} operator for which these definiteness properties are not available. Our contribution is to show that, for time-harmonic Maxwell transmission problems with Lipschitz interfaces, the resulting formulation fits into the operator preconditioning framework \cite{Hiptmair2006}. This framework is particularly appropriate here, since it relies on uniform inf-sup conditions for relevant sesquilinear forms and the duality pairing, rather than on definiteness properties of the induced Galerkin matrices. A main technical step is therefore to identify appropriate discrete spaces and to establish the corresponding uniform inf-sup conditions. Based on these ingredients, we introduce two sesquilinear forms, one involves the Maxwell single layer boundary integral operator, while the other provides a pairing between the discrete spaces. The preconditioner is then built from the associated Galerkin matrices. Assuming that the frequency does not belong to the resonance sets, the spectral condition number of the preconditioned matrix is proved to be uniformly bounded. As the reformulated coupling is equivalent to the Costabel coupling, this yields an efficient method for the discrete system arising from the Costabel coupling formulation.

This paper is organized as follows. \Cref{sec: preliminaries} introduces the functional setting for the volume-boundary variational formulation of the time-harmonic Maxwell transmission problem. In \Cref{sec: problem reformulation and discretization}, the Costabel coupling is reformulated as a variational problem involving only boundary unknowns. The corresponding Galerkin discretization is analyzed, and a uniform discrete inf-sup condition is established. \Cref{sec: operator preconditioning} introduces a block-diagonal preconditioner. Under suitable assumptions, the spectral condition number of the preconditioned matrix is proved to be uniformly bounded. \Cref{sec: numerical results} presents numerical examples. Finally, \Cref{sec: conclusions} contains some concluding remarks.  

\section{Preliminaries}
\label{sec: preliminaries}

This section recalls the basic ingredients needed to define the Costabel coupling formulation of Problem \eqref{eq: scattering problem}, including the relevant function spaces, the  trace operators, the integration by parts formula, and the boundary integral operators. First, for $m\ge 0$ and an open set $\Om\subset \R^3$, denote by $\HHs^m\paren{\Om}$ the Sobolev space of complex vector fields on $\Om$, where $\HHs^0\paren{\Om}=\LLs^2\paren{\Om}$. Let $\LLs^2_{\text{loc}}\paren{\Om}$ denote the Fr{\' e}chet space of complex vector fields that are locally square integrable on $\Om$. For any function $\ub$, let $\overline{\ub}$ denote its complex conjugate. 

The natural spaces for analyzing Problem \eqref{eq: scattering problem} are the following:
\begin{align*}
    \HHs\paren{\curlt, \Om} &:= \left\{\ub \in \LLs^2\paren{\Om} \mid \curlt\, \ub \in \LLs^2\paren{\Om}\right\}, \\
    \HHs_{\text{loc}}\paren{\curlt, \Om} &:= \left\{\ub \in \LLs^2_{\text{loc}}\paren{\Om} \mid \curlt\, \ub \in \LLs^2_{\text{loc}}\paren{\Om}\right\}, \\
    \HHs_{\text{loc}}\paren{\curlt^2, \Om} &:= \left\{\ub \in \HHs_{\text{loc}}\paren{\curlt, \Om} \mid \curlt \, \curlt\, \ub \in \LLs^2_{\text{loc}}\paren{\Om}\right\}, 
\end{align*}
all three spaces are endowed with their natural graph topologies; see \cite[Chap.~4]{HW2021} and \cite[Sect.~2.1]{BH2003}. 

Next, we recall the trace spaces associated with $\HHs\paren{\curlt, \Om_{\text{s}}}$ and $\HHs_{\text{loc}}\paren{\curlt, \Om_0}$. First, consider the bounded, Lipschitz domain $\Om_{\text{s}}$. On the boundary $\Gm = \partial \Om_{\text{s}}$, denote by $\Hs^{-1/2}\paren{\Gm}$ the dual of the Sobolev space $\Hs^{1/2}\paren{\Gm}$. Let $\mathbf{n}\in \LLs^\infty\paren{\Gm}$ be the outward unit normal to $\Gm$, and define the tangential space $\LLs^2_{\tv}\paren{\Gm}:= \left\{\bm{\zeta}\in \LLs^2\paren{\Gm}\mid \bm{\zeta} \cdot \mathbf{n} = \mathbf{0}\right\}$. The Dirichlet trace operator is then defined as
$$
\gm_D^{\text{s}} : \HHs^1\paren{\Om_\text{s}} \to \LLs^2_{\tv}\paren{\Gm}, \qqq \gm_D^{\text{s}}\, \bm{\zeta} := \gm^{\text{s}} \bm{\zeta} \times \mathbf{n},
$$
where $\gm^{\text{s}} : \HHs^1\paren{\Om_\text{s}} \to \LLs^2\paren{\Gm}$ is the classical trace operator. We write $\HHs^{1/2}_\times\paren{\Gm}:= \gm_D^{\text{s}}\paren{\HHs^1\paren{\Om_{\text{s}}}}$ for the range of $\gm_D^{\text{s}}$ in $\LLs^2_{\tv}\paren{\Gm}$. We further define the dual space $\HHs^{-1/2}_\times\paren{\Gm}$ of $\HHs^{1/2}_\times\paren{\Gm}$ through the anti-symmetric pairing
$$
\inprod{\bm{\zeta}}{\bm{\eta}}_{\times, \Gm} := \int_\Gm \paren{\bm{\zeta}\times \mathbf{n}} \cdot \bm{\eta} \ds \qqq \forall \bm{\zeta}, \bm{\eta} \in \LLs^2_{\tv}\paren{\Gm}.
$$
We next introduce the space $\widehat{\Hs}^{3/2}\paren{\Gm}:= \gm^{\text{s}}\paren{\Hs^2\paren{\Om_{\text{s}}}}$ and its dual $\widehat{\Hs}^{-3/2}\paren{\Gm}$. For any $\eta\in \widehat{\Hs}^{3/2}\paren{\Gm}$, there exists $p\in \Hs^2\paren{\Om_{\text{s}}}$ such that $\eta = \gm^{\text{s}} p$. By adopting from the definition of the $\mathbf{grad}_\Gm$ operator \cite[Prop.~3.4]{BCS2002}:
$$
\mathbf{grad}_\Gm : \widehat{\Hs}^{3/2}\paren{\Gm}\to \HHs^{1/2}_\times\paren{\Gm}, \qqq \mathbf{grad}_{\Gm}\, \eta:= \gm_D^{\text{s}}\paren{\nabla p} \times \mathbf{n},
$$
we can define the operator $\operatorname{div}_\Gm: \HHs^{-1/2}_\times\paren{\Gm}\to \widehat{\Hs}^{-3/2}\paren{\Gm}$ as 
$$
\inprod{\operatorname{div}_\Gm \bm{\zeta}}{\eta}_{\widehat{\Hs}^{-3/2}\paren{\Gm},\, \widehat{\Hs}^{3/2}\paren{\Gm}} := -\inprod{\bm{\zeta}}{\mathbf{n}\times \mathbf{grad}_\Gm\, \eta}_{\times, \Gm}\qqq \forall \bm{\zeta} \in \HHs^{-1/2}_\times\paren{\Gm}, \eta\in \widehat{\Hs}^{3/2}\paren{\Gm}.
$$
Consider the space
$$
\HHs^{-1/2}_\times\paren{\operatorname{div}_\Gm, \Gm} := \left\{\bm{\zeta} \in \HHs^{-1/2}_\times\paren{\Gm}\mid \operatorname{div}_\Gm \bm{\zeta} \in \Hs^{-1/2}\paren{\Gm} \right\},
$$
equipped with the graph norm. We recall from \cite[Thm.~2]{BH2003} the following essential result:

\begin{theorem}
    \label{theo: self duality}
    The pairing $\inprod{\cdot}{\cdot}_{\times, \Gm}$ can be extended to a bounded bilinear form on $\HHs^{-1/2}_\times\paren{\operatorname{div}_\Gm, \Gm}$. Moreover, with respect to this pairing, the space $\HHs^{-1/2}_\times\paren{\operatorname{div}_\Gm, \Gm}$ is self-dual.
\end{theorem}

The space $\HHs^{-1/2}_\times\paren{\operatorname{div}_\Gm, \Gm}$ is the natural space for the Dirichlet trace of functions in $\HHs\paren{\curlt, \Om_{\text{s}}}$. The following result can be gathered from \cite[Thm.~4.1]{BCS2002}, \cite[Thm.~5]{Buffa2003}, and \cite[Sect.~3.2]{BHVS2003}:

\begin{theorem}
    \label{theo: integration by parts}
    The Dirichlet trace operator $\gm_D^{\text{s}}$ can be extended to a bounded linear and surjective operator from $\HHs\paren{\curlt, \Om_{\text{s}}}$ into $\HHs^{-1/2}_\times\paren{\operatorname{div}_\Gm, \Gm}$. Moreover, it admits bounded linear right inverses. The following integration by parts formula holds:
    $$
    \int_{\Om_{\text{s}}} \curlt\, \ub \cdot\vb - \ub\cdot\curlt\, \vb \dx = -\inprod{\gm_D^{\text{s}}\, \ub}{\gm_D^{\text{s}}\, \vb}_{\times, \Gm}\qqq \forall \ub, \vb \in \HHs\paren{\curlt, \Om_{\text{s}}}.
    $$
\end{theorem}

Besides the Dirichlet trace operator $\gm_D^{\text{s}}$, the transmission conditions in \eqref{eq: scattering problem} also involve the Neumann trace operator $\gm_N^{\text{s}}$. Its natural domain is
$$
\HHs_{\mu_{\text{s}}}\paren{\curlt^2, \Om_{\text{s}}} := \left\{\ub \in \HHs\paren{\curlt, \Om_{\text{s}}} \mid \curlt \, \mu^{-1}_{\text{s}} \curlt\, \ub \in \LLs^2\paren{\Om_{\text{s}}}\right\}.
$$
In view of the integration by parts in \Cref{theo: integration by parts}, for each $\ub \in \HHs_{\mu_{\text{s}}}\paren{\curlt^2, \Om_{\text{s}}}$, we define $\gm_N^{\text{s}}\, \ub \in \HHs^{-1/2}_\times\paren{\operatorname{div}_\Gm, \Gm}$ as the unique element satisfying
\begin{equation}
    \label{eq: neumann integration by parts}
    \int_{\Om_{\text{s}}} \curlt\, \mu^{-1}_{\text{s}} \curlt\, \ub \cdot\vb - \mu^{-1}_{\text{s}} \curlt\, \ub \cdot\curlt\, \vb \dx =: -\inprod{\gm_N^{\text{s}}\, \ub}{\gm_D^{\text{s}}\, \vb}_{\times, \Gm}\qqq \forall \vb \in \HHs\paren{\curlt, \Om_{\text{s}}}.
\end{equation}
Clearly, $\gm_N^{\text{s}}\, \ub = \gm_D^{\text{s}}\paren{\mu^{-1}_{\text{s}} \curlt\, \ub}$ for all $\ub \in \HHs_{\mu_{\text{s}}}\paren{\curlt^2, \Om_{\text{s}}}$. Moreover, by adapting the proof of \cite[Lem.~3.3]{Hiptmair2002B} to the case of varying $\mu_{\text{s}}$, we obtain a bounded linear operator $\gm_N^{\text{s}}: \HHs_{\mu_{\text{s}}}\paren{\curlt^2, \Om_{\text{s}}} \to \HHs^{-1/2}_\times\paren{\operatorname{div}_\Gm, \Gm}$.

For later use, we also introduce the Dirichlet trace kernel $\HHs_0\paren{\curlt, \Om_{\text{s}}} := \left\{\ub\in \HHs\paren{\curlt, \Om_{\text{s}}}\mid \gm_D^{\text{s}}\, \ub = \bm{0}\right\}$. By \cite[Thm.~3.33]{Monk2003}, this space coincides with the closure of $\CCs_0^\infty\paren{\Om_{\text{s}}}$ in $\HHs\paren{\curlt, \Om_{\text{s}}}$, where $\CCs_0^\infty\paren{\Om_{\text{s}}}$ denotes the space of smooth vector fields compactly supported in $\Om_{\text{s}}$. Fix an arbitrary right inverse $\mathcal{R}$ of $\gm_D^{\text{s}}$. Accordingly, every $\ub \in \HHs\paren{\curlt, \Om_{\text{s}}}$ with Dirichlet trace $\bm{\zeta} = \gm_D^{\text{s}}\, \ub\in \HHs^{-1/2}_\times\paren{\operatorname{div}_\Gm, \Gm}$ admits the decomposition
$$
\ub = \mathcal{R}\, \bm{\zeta} + \paren{\ub - \mathcal{R}\, \bm{\zeta}},
$$
where $\mathcal{R}\, \bm{\zeta}\in \mathcal{R}\paren{\HHs^{-1/2}_\times\paren{\operatorname{div}_\Gm, \Gm}}$ and $\ub - \mathcal{R}\, \bm{\zeta} \in \HHs_0\paren{\curlt, \Om_{\text{s}}}$. Moreover, since $\mathcal{R}\paren{\HHs^{-1/2}_\times\paren{\operatorname{div}_\Gm, \Gm}}$ and $\HHs_0\paren{\curlt, \Om_{\text{s}}}$ intersect trivially, we deduce that
\begin{equation}
    \label{eq: direct sum}
    \HHs\paren{\curlt, \Om_{\text{s}}} = \mathcal{R}\paren{\HHs^{-1/2}_\times\paren{\operatorname{div}_\Gm, \Gm}} \oplus \HHs_0\paren{\curlt, \Om_{\text{s}}}.
\end{equation}

For the exterior domain $\Om_0$, we denote the analogous Dirichlet and Neumann trace operators by $\gm_D^0$ and $\gm_N^0$, respectively. Again, $\HHs^{-1/2}_\times\paren{\operatorname{div}_\Gm, \Gm}$ serves as the space for the Dirichlet trace of fields in $\HHs_{\text{loc}}\paren{\curlt, \Om_0}$ and the Neumann trace of fields in $\HHs_{\text{loc}}\paren{\curlt^2, \Om_0}$.

Finally, we introduce essential boundary integral operators. Let $G_\kappa\paren{\xb, \yb}:= \operatorname{exp}\paren{i\kappa \abs{\xb-\yb}}/ 4\pi\abs{\xb-\yb}$ with $\xb, \yb \in \R^3, \xb\ne \yb$ denote the fundamental solution of the Helmholtz equation. For $\xb \in \Om_{\text{s}} \cup \Om_0$, the Maxwell single and double layer potentials are defined as
\begin{align*}
    \mathbf{SL}_\kappa\paren{\bm{\zeta}}\paren{\xb} &:= \int_\Gm G_\kappa\paren{\xb, \yb} \bm{\zeta}\paren{\yb} \ds \paren{\yb} + \dfrac{1}{\kappa^2}\mathbf{grad}\int_\Gm G_\kappa\paren{\xb, \yb} \operatorname{div}_\Gm \bm{\zeta}\paren{\yb} \ds\paren{\yb}, \\
    \mathbf{DL}_{\kappa}\paren{\bm{\eta}}\paren{\xb} &:= \curlt \int_\Gm G_\kappa\paren{\xb, \yb} \bm{\eta}\paren{\yb} \ds\paren{\yb}.
\end{align*}
These potentials induce continuous linear operators from $\HHs^{-1/2}_\times\paren{\operatorname{div}_\Gm, \Gm}$ into $\HHs_{\text{loc}}\paren{\curlt^2, \Om_{\text{s}} \cup \Om_0}$ \cite[Sect.~3.3]{BHVS2003}. Applying the average trace operator $\left\{\gm_D\right\}_\Gm := \tfrac{1}{2}\paren{\gm_D^{\text{s}} + \gm_D^0}$ to these potentials, we obtain the following continuous linear operators on $\HHs^{-1/2}_\times\paren{\operatorname{div}_\Gm, \Gm}$:
$$
\operatorname{V}_\kappa := \left\{\gm_D\right\}_\Gm \circ \mathbf{SL}_\kappa,\qqq \operatorname{K}_\kappa := \left\{\gm_D\right\}_\Gm \circ \mathbf{DL}_\kappa,
$$
see \cite[Cor.~2]{BH2003}. The operators $\operatorname{V}_\kappa$ and $\operatorname{K}_\kappa$ are called the Maxwell single and double layer boundary integral operators, respectively. Let $\ub \in \HHs_{\text{loc}}\paren{\curlt^2, \Om_0}$ be the solution of the exterior homogeneous Maxwell problem
$$
\curlt\, \mu_0^{-1}\curlt\, \ub - \omega^2 \varepsilon_0\, \ub = \bm{0} \qqq \text{in } \Om_0,
$$
subject to the Silver-M{\" u}ller radiation condition. Then the exterior Calder{\' o}n identities read as follows:
\begin{equation}
\label{eq: calderon identities}
    \begin{alignedat}{2}
        \gm_D^0\, \ub &= \paren{\tfrac{1}{2}\operatorname{Id} - \operatorname{K}_\kappa} \paren{\gm_D^0\, \ub} & &- \mu_0 \operatorname{V}_\kappa \paren{\gm_N^0\, \ub}, \\
        \gm_N^0\, \ub &= -\om^2\varepsilon_0 \operatorname{V}_\kappa \paren{\gm_D^0\, \ub} & &+ \paren{\tfrac{1}{2}\operatorname{Id} - \operatorname{K}_\kappa} \paren{\gm_N^0\, \ub}.
    \end{alignedat}
\end{equation}
where $\operatorname{Id}$ denotes the identity operator on $\HHs^{-1/2}_\times\paren{\operatorname{div}_\Gm, \Gm}$; see \cite[Sect.~5]{BH2003} and \cite[Sect.~3.3]{BHVS2003}.

Throughout this paper, $C>0$ denotes a generic constant, which may vary from line to line, depending only on $\Om_{\text{s}}, \mu_{\text{s}}, \varepsilon_{\text{s}}$ and $\kappa$, but not on $\EEs^{\text{inc}}$ or the solution $\paren{\EEs, \bm{\uplambda}}$ of Problem \eqref{eq: old formulation}.

\section{Problem reformulation and discretization}
\label{sec: problem reformulation and discretization}

In this section, we begin with the Costabel coupling formulation of Problem \eqref{eq: scattering problem}, derive an equivalent variational formulation, and then introduce its discrete counterpart. 

\subsection{Problem reformulation}
\label{subsec: problem reformulation}

Assume that Problem \eqref{eq: scattering problem} admits a solution $\EEs\in \HHs_{\text{loc}}\paren{\curlt, \Om_{\text{s}} \cup \Om_0}$ such that $\EEs_{\mid\, \Om_{\text{s}}}\in \HHs_{\mu_{\text{s}}}\paren{\curlt^2, \Om_{\text{s}}}$ and $\EEs_{\mid\, \Om_0} \in \HHs_{\text{loc}}\paren{\curlt^2, \Om_0}$. For this solution, \eqref{eq: neumann integration by parts} yields
$$
\int_{\Om_{\text{s}}} \mu_{\text{s}}^{-1}\curlt \EEs \cdot\, \curlt \, \overline{\WWs} - \omega^2 \varepsilon_{\text{s}}\EEs\cdot\, \overline{\WWs} \dx = \inprod{\gm_N^{\text{s}} \EEs}{\gm_D^{\text{s}}\overline{\WWs}}_{\times, \Gm}\qqq \forall \WWs\in \HHs\paren{\curlt, \Om_{\text{s}}}.
$$
Owing to \eqref{eq: calderon identities} and the transmission conditions in \eqref{eq: scattering problem}, we have
\begin{align*}
    \gm_N^{\text{s}} \EEs = \gm_N^0 \EEs + \gm_N^0 \EEs^{\text{inc}} &= -\om^2\varepsilon_0 \operatorname{V}_\kappa \paren{\gm_D^0 \EEs} + \paren{\tfrac{1}{2}\operatorname{Id} - \operatorname{K}_\kappa} \paren{\gm_N^0 \EEs} + \gm_N^0 \EEs^{\text{inc}} \\
    & = -\om^2\varepsilon_0 \operatorname{V}_\kappa \paren{\gm_D^{\text{s}} \EEs} + \paren{\tfrac{1}{2}\operatorname{Id} - \operatorname{K}_\kappa} \paren{\gm_N^0 \EEs} + \om^2\varepsilon_0 \operatorname{V}_\kappa \paren{\gm_D^0 \EEs^{\text{inc}}} +  \gm_N^0 \EEs^{\text{inc}},
\end{align*}
and
$$
\paren{\tfrac{1}{2}\operatorname{Id} + \operatorname{K}_\kappa} \paren{\gm_D^{\text{s}} \EEs} + \mu_0 \operatorname{V}_\kappa \paren{\gm_N^0 \EEs} = \paren{\tfrac{1}{2}\operatorname{Id} + \operatorname{K}_\kappa} \paren{\gm_D^0 \EEs^{\text{inc}}}.
$$
We now combine the last three equations and replace $\gm_N^0 \EEs$ with an independent unknown $\bm{\uplambda}$. The Costabel coupling formulation \cite[Sect.~6]{Hiptmair2003}, in the present setting, reads as follows: Determine $\paren{\EEs, \bm{\uplambda}}\in \HHs\paren{\curlt, \Om_{\text{s}}}\times \HHs^{-1/2}_{\times}\paren{\operatorname{div}_{\Gm}, \Gm}$ such that
\begin{equation}
    \label{eq: old formulation}
    \left\{\begin{alignedat}{3}
         &a\paren{\EEs, \WWs} + \om^2\varepsilon_0\inprod{\operatorname{V}_{\kappa}\paren{\gm_D^{\text{s}} \EEs}}{\gm_D^{\text{s}} \overline{\WWs}}_{\times, \Gm} \ &+&\  \inprod{\paren{-\tfrac{1}{2}\operatorname{Id} + \operatorname{K}_\kappa}\bm{\uplambda}}{\gm_D^{\text{s}}\overline{\WWs}}_{\times, \Gm} \ &=&\  \ell_1\paren{\WWs}, \\
         &\inprod{\paren{\tfrac{1}{2}\operatorname{Id} + \operatorname{K}_\kappa}\paren{\gm_D^{\text{s}}\EEs}}{\overline{\bm{\varphi}}}_{\times, \Gm} \ &+&\ \mu_0\inprod{\operatorname{V}_\kappa \bm{\uplambda}}{\overline{\bm{\varphi}}}_{\times, \Gm} \ &=&\ \ell_2\paren{\bm{\vphi}},
    \end{alignedat}\right.
\end{equation}
for all $\paren{\WWs, \bm{\varphi}}\in \HHs\paren{\curlt, \Om_{\text{s}}}\times \HHs^{-1/2}_{\times}\paren{\operatorname{div}_{\Gm}, \Gm}$, where the sesquilinear form $a : \HHs\paren{\curlt, \Om_{\text{s}}}\times \HHs\paren{\curlt, \Om_{\text{s}}} \to \C$ is defined as
$$
a\paren{\ub, \vb} := \int_{\Om_{\text{s}}} \mu_{\text{s}}^{-1}\curlt\, \ub \cdot \curlt\, \overline{\vb} - \omega^2 \varepsilon_{\text{s}}\, \ub\cdot \overline{\vb} \dx \qqq \forall \ub, \vb \in \HHs\paren{\curlt, \Om_{\text{s}}},
$$
and the linear forms $\ell_1: \HHs\paren{\curlt, \Om_{\text{s}}}\to \C$ and $\ell_2: \HHs^{-1/2}_{\times}\paren{\operatorname{div}_{\Gm}, \Gm}\to \C$ are given by
\begin{align*}
    \ell_1\paren{\vb} & := \om^2\varepsilon_0\inprod{\operatorname{V}_{\kappa}\paren{\gm_D^0 \EEs^{\text{inc}}}}{\gm_D^{\text{s}}\, \overline{\vb}}_{\times, \Gm} + \inprod{\gm_N^0 \EEs^{\text{inc}}}{\gm_D^{\text{s}}\, \overline{\vb}}_{\times, \Gm} &\forall& \vb \in \HHs\paren{\curlt, \Om_{\text{s}}}, \\
    \ell_2\paren{\bm{\eta}} & := \inprod{\paren{\tfrac{1}{2}\operatorname{Id} + \operatorname{K}_\kappa}\paren{\gm_D^0 \EEs^{\text{inc}}}}{\overline{\bm{\eta}}}_{\times, \Gm} &\forall& \bm{\eta} \in \HHs^{-1/2}_{\times}\paren{\operatorname{div}_{\Gm}, \Gm}.
\end{align*}
Here, $\EEs^{\text{inc}}$ is an incident electric field in $\Om_0$ such that $\left.\EEs^{\text{inc}}\right._{\mid\, D} \in \HHs\paren{\curlt^2, D}$, where $D\subset \Om_0$ is a bounded Lipschitz domain containing an exterior neighborhood of $\Gm$. We next state two assumptions on the angular frequency $\om$. The first one is used to establish the well-posedness of Problem \eqref{eq: old formulation}, whereas the second will only be invoked later for the lifting problem.

\begin{assumption}
    \label{assum: resonant frequency}
    The angular frequency $\omega$ is assumed to lie outside the resonance sets of the source-free constant-coefficient and variable-coefficient Maxwell problems in $\Om_{\text{s}}$. More precisely, we assume that:
    \begin{enumerate}[label=(\roman*), ref=\roman*]
        \item\label{assum: resonant frequency:const}
        The problem $\curlt\, \mu_0^{-1}\curlt\, \ub - \omega^2 \varepsilon_0\, \ub = \bm{0}$ in $\Om_{\text{s}}$ subject to $\gm_D^{\text{s}} \, \ub = \bm{0}$ on $\Gm$ admits only a trivial solution.
        \item\label{assum: resonant frequency:var}
        The problem $\curlt\, \mu_{\text{s}}^{-1}\curlt\, \ub - \omega^2 \varepsilon_{\text{s}}\, \ub = \bm{0}$ in $\Om_{\text{s}}$ subject to $\gm_D^{\text{s}} \, \ub = \bm{0}$ on $\Gm$ admits only a trivial solution.
    \end{enumerate}
\end{assumption}

The left-hand side of Problem \eqref{eq: old formulation} defines a sesquilinear form that is $T$-coercive up to a compact perturbation \cite[Thm.~8.3]{Hiptmair2003}. Under \Cref{assum: resonant frequency}(\ref{assum: resonant frequency:const}), the associated operator is injective \cite[Thm.~1.1, Lem.~6.1]{Hiptmair2003}. The well-posedness of Problem \eqref{eq: old formulation} therefore follows from the Riesz--Schauder theorem \cite[Thm.~5.3.7]{HW2021}.

\begin{proposition}
    \label{prop: old formulation well-posed}
    Under \Cref{assum: resonant frequency}(\ref{assum: resonant frequency:const}), Problem \eqref{eq: old formulation} is well-posed. 
\end{proposition}

We now turn to the derivation of an equivalent variational problem for \eqref{eq: old formulation} on $\HHs^{-1/2}_{\times}\paren{\operatorname{div}_{\Gm}, \Gm}\times \HHs^{-1/2}_{\times}\paren{\operatorname{div}_{\Gm}, \Gm}$. Since $a\paren{\EEs, \WWs}$ is the only volumetric contribution on the left-hand side of \eqref{eq: old formulation}, this derivation hinges on expressing this term by means of the pairing $\inprod{\cdot}{\cdot}_{\times, \Gm}$. Accordingly, we introduce the trace lifting operator $\mathcal{E} : \HHs^{-1/2}_{\times}\paren{\operatorname{div}_{\Gm}, \Gm} \to \HHs\paren{\curlt, \Om_{\text{s}}}$ which assigns to each $\bm{\zeta} \in \HHs^{-1/2}_{\times}\paren{\operatorname{div}_{\Gm}, \Gm}$ the solution $\mathcal{E}\, \bm{\zeta}\in \HHs\paren{\curlt, \Om_{\text{s}}}$ of the lifting problem
\begin{equation}
    \label{eq: trace lifting operator}
    \left\{\begin{alignedat}{2}
    &a\paren{\mathcal{E}\, \bm{\zeta} , \vb}\  &=&\ \bm{0} \q \forall \vb \in \HHs_0\paren{\curlt, \Om_{\text{s}}}, \\
    &\gm_D^{\text{s}} \paren{\mathcal{E}\, \bm{\zeta}}\  &=&\ \bm{\zeta}.
    \end{alignedat}
    \right.
\end{equation}

\begin{lemma}
    \label{lem: well-defined of trace lifting operator}
    Assume that \Cref{assum: resonant frequency}(\ref{assum: resonant frequency:var}) holds. For each $\bm{\zeta} \in \HHs^{-1/2}_{\times}\paren{\operatorname{div}_{\Gm}, \Gm}$, Problem \eqref{eq: trace lifting operator} admits a unique solution. Moreover, $\mathcal{E}$ defines a bounded linear operator. 
\end{lemma}

\begin{proof}
    Using a fixed bounded right inverse $\mathcal{R}$ of $\gm_D^{\text{s}}$, we write $\mathcal{E}\, \bm{\zeta} = \mathcal{R}\, \bm{\zeta} + \UUs_{\bm{\zeta}}$ with $\UUs_{\bm{\zeta}} \in \HHs_0\paren{\curlt, \Om_{\text{s}}}$. Then Problem \eqref{eq: trace lifting operator} is equivalent to 
    $$
    a\paren{\UUs_{\bm{\zeta}}, \vb} = \ell_{\bm{\zeta}}\paren{\vb} \qqq \forall \vb \in \HHs_0\paren{\curlt, \Om_{\text{s}}},
    $$
    where $\ell_{\bm{\zeta}}\paren{\vb} := - a\paren{\mathcal{R}\, \bm{\zeta}, \vb}$. Under \Cref{assum: resonant frequency}(\ref{assum: resonant frequency:var}), this problem admits a unique solution $\UUs_{\bm{\zeta}} \in \HHs_0\paren{\curlt, \Om_{\text{s}}}$ that satisfies
    $$
    \norm{\UUs_{\bm{\zeta}}}_{\HHs_0\paren{\curlt, \Om_{\text{s}}}}\le C \norm{\ell_{\bm{\zeta}}}_{\HHs_0\paren{\curlt, \Om_{\text{s}}}^\prime},
    $$
    see \cite[Thm.~5.2]{Hiptmair2002}. By the continuity of $a$ and the boundedness of $\mathcal{R}$, we have 
    $$
    \norm{\ell_{\bm{\zeta}}}_{\HHs_0\paren{\curlt, \Om_{\text{s}}}^\prime}\le C \norm{\mathcal{R}\, \bm{\zeta}}_{\HHs\paren{\curlt, \Om_{\text{s}}}} \le C \norm{\bm{\zeta}}_{\HHs^{-1/2}_{\times}\paren{\operatorname{div}_{\Gm}, \Gm}},
    $$
    and hence $\norm{\mathcal{E}\, \bm{\zeta}}_{\HHs\paren{\curlt, \Om_{\text{s}}}}\le C \norm{\bm{\zeta}}_{\HHs^{-1/2}_{\times}\paren{\operatorname{div}_{\Gm}, \Gm}}$ for all $\bm{\zeta}\in \HHs^{-1/2}_{\times}\paren{\operatorname{div}_{\Gm}, \Gm}$. The proof is complete.
\end{proof}

By taking $\vb\in \CCs_0^\infty\paren{\Om_{\text{s}}}$ in \eqref{eq: trace lifting operator}, we infer that
$$
\curlt\, \mu_{\text{s}}^{-1}\curlt\, \mathcal{E}\, \bm{\zeta} = \omega^2 \varepsilon_{\text{s}}\, \mathcal{E}\, \bm{\zeta} \qqq \text{in } \CCs_0^\infty\paren{\Om_{\text{s}}}^\prime.
$$
Since $\om^2\varepsilon_{\text{s}}\, \mathcal{E}\, \bm{\zeta} \in \LLs^2\paren{\Om_{\text{s}}}$, this implies $\curlt\, \mu_{\text{s}}^{-1}\curlt\, \mathcal{E}\, \bm{\zeta} \in \LLs^2\paren{\Om_{\text{s}}}$, and hence $\mathcal{E}\, \bm{\zeta} \in \HHs_{\mu_{\text{s}}}\paren{\curlt^2, \Om_{\text{s}}}$. Owing to \eqref{eq: neumann integration by parts} and the above distributional equation, we have
\begin{align}
    a\paren{\mathcal{E}\, \bm{\zeta}, \vb}
    &= \int_{\Om_{\text{s}}} \mu_{\text{s}}^{-1}\curlt\, \mathcal{E}\, \bm{\zeta} \cdot \curlt\, \overline{\vb}
       - \omega^2 \varepsilon_{\text{s}}\, \mathcal{E}\, \bm{\zeta}\cdot \overline{\vb} \dx
       \notag\\
    &= \int_{\Om_{\text{s}}} \curlt\, \mu_{\text{s}}^{-1}\curlt\, \mathcal{E}\, \bm{\zeta} \cdot \overline{\vb}
       - \omega^2 \varepsilon_{\text{s}}\, \mathcal{E}\, \bm{\zeta}\cdot \overline{\vb} \dx
       + \inprod{\gm_N^{\text{s}} \, \mathcal{E}\, \bm{\zeta}}{\gm_D^{\text{s}}\, \overline{\vb}}_{\times, \Gm}
       \label{eq: integration by part with trace lifting}
       \displaybreak[2]\\
    &= \inprod{\gm_N^{\text{s}} \, \mathcal{E}\, \bm{\zeta}}{\gm_D^{\text{s}}\, \overline{\vb}}_{\times, \Gm},
       \notag
\end{align}
for all $\vb \in \HHs\paren{\curlt, \Om_{\text{s}}}$. Now let $\mathcal{R}$ be an arbitrary right inverse of $\gm_D^{\text{s}}$, and define
\begin{equation}
    \label{eq: Steklov--Poincare operator}
    \mathcal{S}: \HHs^{-1/2}_{\times}\paren{\operatorname{div}_{\Gm}, \Gm}\to  \HHs^{-1/2}_{\times}\paren{\operatorname{div}_{\Gm}, \Gm},\qqq \inprod{\mathcal{S}\, \bm{\zeta}}{\overline{\bm{\eta}}}_{\times, \Gm}:= a\paren{\mathcal{E}\, \bm{\zeta}, \mathcal{R}\, \bm{\eta}} \qqq \forall \bm{\eta}\in \HHs^{-1/2}_{\times}\paren{\operatorname{div}_{\Gm}, \Gm}.
\end{equation}
This definition is independent of the choice of $\mathcal{R}$. Indeed, if $\mathcal{R}_1$ and $\mathcal{R}_2$ be two different right inverses of $\gm_D^{\text{s}}$, then $\mathcal{R}_1\, \bm{\eta} - \mathcal{R}_2\, \bm{\eta} \in \HHs_0\paren{\curlt, \Om_{\text{s}}}$ for all $\bm{\eta} \in \HHs^{-1/2}_{\times}\paren{\operatorname{div}_{\Gm}, \Gm}$, and hence, by \eqref{eq: trace lifting operator}, it follows that $a\paren{\mathcal{E}\, \bm{\zeta}, \mathcal{R}_1\, \bm{\eta} -\mathcal{R}_2\, \bm{\eta}} = 0$. Therefore $a\paren{\mathcal{E}\, \bm{\zeta}, \mathcal{R}_1\, \bm{\eta}} = a\paren{\mathcal{E}\, \bm{\zeta}, \mathcal{R}_2\, \bm{\eta}}$. By choosing $\vb = \mathcal{R}\, \bm{\eta}$ in \eqref{eq: integration by part with trace lifting}, we further obtain
$$
\inprod{\mathcal{S}\, \bm{\zeta}}{\overline{\bm{\eta}}}_{\times, \Gm} = \inprod{\gm_N^{\text{s}} \, \mathcal{E}\, \bm{\zeta}}{\overline{\bm{\eta}}}_{\times, \Gm}\qqq \forall \bm{\zeta}, \bm{\eta} \in \HHs^{-1/2}_{\times}\paren{\operatorname{div}_{\Gm}, \Gm},
$$
that is, $\mathcal{S}\, \bm{\zeta}$ coincides with the Neumann trace of $\mathcal{E}\, \bm{\zeta}$. In this sense, $\mathcal{S}$ is referred to as the Steklov--Poincar{\' e} operator. Other variants of the Steklov--Poincar{\' e} operator can be found in \cite[Sect.~4]{Bourdonnaye1995} and \cite[Def.~2.1]{DZH+2015}. 

Since any right inverse may be used in \eqref{eq: Steklov--Poincare operator}, we may take $\mathcal{R}=\mathcal{E}$ to have $\inprod{\mathcal{S}\, \bm{\zeta}}{\overline{\bm{\eta}}}_{\times, \Gm}= a\paren{\mathcal{E}\, \bm{\zeta}, \mathcal{E}\, \bm{\eta}}$. From this, we see that $\mathcal{S}$ induces on $\HHs^{-1/2}_{\times}\paren{\operatorname{div}_{\Gm}, \Gm}$ a Hermitian sesquilinear form, which is in general neither positive definite nor positive semi-definite.

Owing to \eqref{eq: Steklov--Poincare operator}, the term $a\paren{\EEs, \WWs}$ on the left-hand side of \eqref{eq: old formulation} can now be rewritten in terms of the pairing $\inprod{\cdot}{\cdot}_{\times, \Gm}$, involving the operator $\mathcal{S}$ and the Dirichlet traces of $\EEs$ and $\WWs$. This leads to the following variational problem: Determine $\paren{\bm{\tau}, \bm{\uplambda}}\in \HHs^{-1/2}_{\times}\paren{\operatorname{div}_{\Gm}, \Gm}\times \HHs^{-1/2}_{\times}\paren{\operatorname{div}_{\Gm}, \Gm}$ that satisfies
\begin{equation}
    \label{eq: new formulation}
    \left\{\begin{alignedat}{3}
         &\inprod{\paren{\mathcal{S} + \om^2\varepsilon_0 \operatorname{V}_\kappa}\bm{\tau}}{\overline{\bm{\xi}}}_{\times, \Gm} \ &+&\  \inprod{\paren{-\tfrac{1}{2}\operatorname{Id} + \operatorname{K}_\kappa}\bm{\uplambda}}{\overline{\bm{\xi}}}_{\times, \Gm} \ &=&\  \ell_{1, \Gm}\paren{\bm{\xi}}, \\
         &\inprod{\paren{\tfrac{1}{2}\operatorname{Id} + \operatorname{K}_\kappa}\bm{\tau}}{\overline{\bm{\varphi}}}_{\times, \Gm} \ &+&\ \mu_0\inprod{\operatorname{V}_\kappa \bm{\uplambda}}{\overline{\bm{\varphi}}}_{\times, \Gm} \ &=&\ \ell_2\paren{\bm{\vphi}},
    \end{alignedat}\right.
\end{equation}
for all $\paren{\bm{\xi}, \bm{\varphi}}\in \HHs^{-1/2}_{\times}\paren{\operatorname{div}_{\Gm}, \Gm}\times \HHs^{-1/2}_{\times}\paren{\operatorname{div}_{\Gm}, \Gm}$. Here, the linear form $\ell_{1,\Gm}: \HHs^{-1/2}_{\times}\paren{\operatorname{div}_{\Gm}, \Gm} \to \Cs$ is given by
$$
\ell_{1,\Gm}\paren{\bm{\zeta}} := \om^2\varepsilon_0\inprod{\operatorname{V}_{\kappa}\paren{\gm_D^0 \EEs^{\text{inc}}}}{\overline{\bm{\zeta}}}_{\times, \Gm} + \inprod{\gm_N^0 \EEs^{\text{inc}}}{\overline{\bm{\zeta}}}_{\times, \Gm} \qqq \forall \bm{\zeta} \in \HHs^{-1/2}_{\times}\paren{\operatorname{div}_{\Gm}, \Gm}.
$$
Moreover, $\ell_{1,\Gm}\paren{\gm_D^{\text{s}}\, \vb} = \ell_1\paren{\vb}$ for all $\vb \in \HHs\paren{\curlt, \Om_{\text{s}}}$. We have the following result:

\begin{theorem}
    \label{theo: continuous problems equivalence}
    Under \Cref{assum: resonant frequency}, Problems \eqref{eq: old formulation} and \eqref{eq: new formulation} are equivalent. 
\end{theorem}

\begin{proof}
    Let $\paren{\EEs, \bm{\uplambda}}\in \HHs\paren{\curlt, \Om_{\text{s}}}\times \HHs^{-1/2}_{\times}\paren{\operatorname{div}_{\Gm}, \Gm}$ be the solution of Problem \eqref{eq: old formulation}. By testing \eqref{eq: old formulation} with $\WWs \in \HHs_0\left(\curlt, \Om_{\text{s}}\right)$, we obtain
    $$
    a\paren{\EEs, \WWs} = \bm{0} \qqq \forall \WWs \in \HHs_0\left(\curlt, \Om_{\text{s}}\right),
    $$
    which means that $\EEs = \mathcal{E}\, \bm{\tau}$, where $\bm{\tau} = \gm_D^{\text{s}} \EEs$. Now, fix $\WWs\in \HHs\paren{\curlt, \Om_{\text{s}}}$, and denote $\bm{\xi} = \gm_D^{\text{s}} \WWs$. If $\mathcal{R}$ is an arbitrary right inverse of $\gm_D^{\text{s}}$, then $\WWs-\mathcal{R}\, \bm{\xi} \in \HHs_0\paren{\curlt, \Om_{\text{s}}}$. It thus follows from \eqref{eq: trace lifting operator} and \eqref{eq: Steklov--Poincare operator} that $a\paren{\EEs, \WWs} = a\paren{\mathcal{E}\, \bm{\tau}, \mathcal{R}\, \bm{\xi}} = \inprod{\mathcal{S}\, \bm{\tau}}{\overline{\bm{\xi}}}_{\times, \Gm}$. Substituting this into \eqref{eq: old formulation}, we obtain
    $$
    \left\{\begin{alignedat}{3}
         &\inprod{\paren{\mathcal{S} + \om^2\varepsilon_0 \operatorname{V}_\kappa}\bm{\tau}}{\overline{\bm{\xi}}}_{\times, \Gm} \ &+&\  \inprod{\paren{-\tfrac{1}{2}\operatorname{Id} + \operatorname{K}_\kappa}\bm{\uplambda}}{\overline{\bm{\xi}}}_{\times, \Gm} \ &=&\  \ell_{1,\Gm}\paren{\bm{\xi}}, \\
         &\inprod{\paren{\tfrac{1}{2}\operatorname{Id} + \operatorname{K}_\kappa}\bm{\tau}}{\overline{\bm{\varphi}}}_{\times, \Gm} \ &+&\ \mu_0\inprod{\operatorname{V}_\kappa \bm{\uplambda}}{\overline{\bm{\varphi}}}_{\times, \Gm} \ &=&\ \ell_2\paren{\bm{\vphi}},
    \end{alignedat}\right.
    $$
    for all $\bm{\varphi}\in \HHs^{-1/2}_{\times}\paren{\operatorname{div}_{\Gm}, \Gm}$. By varying $\WWs$ in $\HHs\paren{\curlt, \Om_{\text{s}}}$ and invoking the surjectivity of $\gm_D^{\text{s}}$, we recover Problem \eqref{eq: new formulation}. Therefore, Problem \eqref{eq: new formulation} admits at least one solution $\paren{\bm{\tau}, \bm{\uplambda}} = \paren{\gm_D^{\text{s}} \EEs, \bm{\uplambda}}$.

    Conversely, let $\paren{\bm{\tau}, \bm{\uplambda}}\in \HHs^{-1/2}_{\times}\paren{\operatorname{div}_{\Gm}, \Gm}\times \HHs^{-1/2}_{\times}\paren{\operatorname{div}_{\Gm}, \Gm}$ be a solution of Problem \eqref{eq: new formulation}. From $\bm{\tau} \in \HHs^{-1/2}_{\times}\paren{\operatorname{div}_{\Gm}, \Gm}$, one can uniquely determine $\mathcal{E}\, \bm{\tau} \in \HHs\paren{\curlt, \Om_{\text{s}}}$ using \eqref{eq: trace lifting operator}. Consider a fixed $\bm{\xi} \in \HHs^{-1/2}_{\times}\paren{\operatorname{div}_{\Gm}, \Gm}$. By substituting \eqref{eq: Steklov--Poincare operator} into \eqref{eq: new formulation}, we deduce that
    $$
    \left\{\begin{alignedat}{3}
        &a\paren{\mathcal{E}\, \bm{\tau}, \mathcal{R}\, \bm{\xi}} + \om^2\varepsilon_0\inprod{\operatorname{V}_{\kappa}\paren{\gm_D^{\text{s}}\, \mathcal{E}\, \bm{\tau}}}{\gm_D^{\text{s}}\, \overline{\mathcal{R}\, \bm{\xi}}}_{\times, \Gm} \ &+&\  \inprod{\paren{-\tfrac{1}{2}\operatorname{Id} + \operatorname{K}_\kappa}\bm{\uplambda}}{\gm_D^{\text{s}}\, \overline{\mathcal{R}\, \bm{\xi}}}_{\times, \Gm} \ &=&\  \ell_{1,\Gm}\paren{\bm{\xi}}, \\
        &\inprod{\paren{\tfrac{1}{2}\operatorname{Id} + \operatorname{K}_\kappa}\paren{\gm_D^{\text{s}}\, \mathcal{E}\, \bm{\tau}}}{\overline{\bm{\varphi}}}_{\times, \Gm} \ &+&\ \mu_0\inprod{\operatorname{V}_\kappa \bm{\uplambda}}{\overline{\bm{\varphi}}}_{\times, \Gm} \ &=&\ \ell_2\paren{\bm{\vphi}},
    \end{alignedat}\right.
    $$
    for all $\bm{\varphi}\in \HHs^{-1/2}_{\times}\paren{\operatorname{div}_{\Gm}, \Gm}$. Letting $\bm{\xi}$ vary in $\HHs^{-1/2}_{\times}\paren{\operatorname{div}_{\Gm}, \Gm}$, we obtain a coupling formulation with the test space $\mathcal{R}\paren{\HHs^{-1/2}_\times\paren{\operatorname{div}_\Gm, \Gm}} \times \HHs^{-1/2}_{\times}\paren{\operatorname{div}_{\Gm}, \Gm}$. Combining this with \eqref{eq: trace lifting operator}, and using \eqref{eq: direct sum}, we recover Problem \eqref{eq: old formulation} with $\EEs = \mathcal{E}\, \bm{\tau}$ and $\WWs = \mathcal{R}\, \bm{\xi} + \vb$, where $\vb\in \HHs_0\paren{\curlt, \Om_{\text{s}}}$.
\end{proof}

\begin{remark}
    \label{rem: difference with Dobbelaere}
    Problem \eqref{eq: new formulation} has a form similar to that of \cite[Eq. (41)]{DZH+2015}, and both are posed on $\HHs^{-1/2}_{\times}\paren{\operatorname{div}_{\Gm}, \Gm}\times \HHs^{-1/2}_{\times}\paren{\operatorname{div}_{\Gm}, \Gm}$. However, the two formulations are conceptually different, as the corresponding unknowns have different meanings. Indeed, by \Cref{theo: continuous problems equivalence} and \cite[Lem.~6.1]{Hiptmair2003}, Problem \eqref{eq: new formulation} may be formally viewed as determining $\paren{\gm_D^{\text{s}} \EEs, \gm_N^0 \EEs}$. In contrast, \cite[Eq. (41)]{DZH+2015} finds $\paren{-\gm_D^{\text{s}} \EEs, \mathbf{m}_0}$, where $\mathbf{m}_0$ is not the Neumann trace itself, but an auxiliary variable entering the integral representation of $\EEs_{\mid \, \Om_0}$.
\end{remark}

\subsection{Finite element discretization}
\label{subsec: finite element discretization}

Assume that $\Om_{\text{s}}$ is a polyhedron so that $\Gm$ is a polyhedral surface. Consider a family of tetrahedral, quasi-uniform meshes $\left\{\Om_h\right\}_{h\in \paren{0, h_{\#}}}$ of $\Om_{\text{s}}$, where the mesh size $h\in \paren{0, h_{\#}}$ for a given constant $h_{\#}>0$. Restricting $\left\{\Om_h\right\}_{h\in \paren{0, h_{\#}}}$ to $\Gm$ yields a family of surface meshes $\left\{\Gm_h\right\}_{h\in \paren{0, h_{\#}}}$, whose shape-regularity is inherited from $\left\{\Om_h\right\}_{h\in \paren{0, h_{\#}}}$. Denote by $\mathcal{K}_h$ the set of tetrahedra in $\Om_h$, and by $\mathcal{F}_h$ the set of all triangles in $\Gm_h$. In addition to the dependencies specified in \Cref{sec: preliminaries}, the generic constant $C>0$ may depend on the shape-regularity of $\left\{\Om_h\right\}_{h\in \paren{0, h_{\#}}}$, but not on $h$. 

The finite element spaces are defined as
$$
\begin{aligned}
    \XXs_h :=& \left\{\ub_h\in \HHs\paren{\curlt, \Om_{\text{s}}}\mid \ub_{h\, \mid\, K} = \bm{a} + \bm{b}\times \xb \text{ with } \bm{a}, \bm{b}\in \C^3, \xb \in \R^3 \text{ for all }K\in \mathcal{K}_h\right\},\\
    \YYs_h := \gm_D^{\text{s}} \XXs_h =& \left\{\bm{\zeta}_h\in \HHs^{-1/2}_{\times}\paren{\operatorname{div}_{\Gm}, \Gm} \mid \bm{\zeta}_{h\, \mid\, F} = \bm{c} + d\hat{\xb} \text{ with } \bm{c}\in \C^2, d\in \C, \hat{\xb} \in \R^2 \text{ for all } F\in \mathcal{F}_h\right\}.
\end{aligned}
$$
Here, $\XXs_h$ denotes the space of lowest-order N{\' e}d{\' e}lec elements in $\Om_h$, whereas $\YYs_h$ is the associated space of lowest-order Raviart--Thomas elements on $\Gm_h$ \cite[Sect.~5.4 and 5.5]{Monk2003}. Following \cite[Sect.~9]{Hiptmair2003}, the discrete counterpart of Problem \eqref{eq: old formulation} is given by: Determine $\paren{\EEs_h, \bm{\uplambda}_h}\in \XXs_h\times \YYs_h$ such that
\begin{equation}
    \label{eq: discrete old formulation}
    \left\{\begin{alignedat}{3}
         &a\paren{\EEs_h, \WWs_h} + \om^2\varepsilon_0\inprod{\operatorname{V}_{\kappa}\paren{\gm_D^{\text{s}} \EEs_h}}{\gm_D^{\text{s}} \overline{\WWs_h}}_{\times, \Gm} \ &+&\  \inprod{\paren{-\tfrac{1}{2}\operatorname{Id} + \operatorname{K}_\kappa}\bm{\uplambda}_h}{\gm_D^{\text{s}}\overline{\WWs_h}}_{\times, \Gm} \ &=&\  \ell_1\paren{\WWs_h}, \\
         &\inprod{\paren{\tfrac{1}{2}\operatorname{Id} + \operatorname{K}_\kappa}\paren{\gm_D^{\text{s}}\EEs_h}}{\overline{\bm{\varphi}_h}}_{\times, \Gm} \ &+&\ \mu_0\inprod{\operatorname{V}_\kappa \bm{\uplambda}_h}{\overline{\bm{\varphi}_h}}_{\times, \Gm} \ &=&\ \ell_2\paren{\bm{\vphi}_h},
    \end{alignedat}\right.
\end{equation}
for all $\paren{\WWs_h, \bm{\varphi}_h}\in \XXs_h\times \YYs_h$. Let $\bm{\mathcal{B}}_{\text{C}} : \paren{\XXs_h\times \YYs_h}\times \paren{\XXs_h\times \YYs_h}\to \C$ denote the sesquilinear form induced by the left-hand side of \eqref{eq: discrete old formulation}, and endow $\HHs\paren{\curlt, \Om_{\text{s}}}\times \HHs^{-1/2}_{\times}\paren{\operatorname{div}_{\Gm}, \Gm}$ with the canonical product norm
$$
\norm{\paren{\ub, \bm{\zeta}}}_{\Om_{\text{s}}, \Gm}^2 := \norm{\ub}^2_{\HHs\paren{\curlt, \Om_{\text{s}}}} + \norm{\bm{\zeta}}^2_{\HHs^{-1/2}_{\times}\paren{\operatorname{div}_{\Gm}, \Gm}} \qqq \forall \paren{\ub, \bm{\zeta}}\in \HHs\paren{\curlt, \Om_{\text{s}}}\times \HHs^{-1/2}_{\times}\paren{\operatorname{div}_{\Gm}, \Gm}.
$$
Under \Cref{assum: resonant frequency}(\ref{assum: resonant frequency:const}), \cite[Thm.~11.1]{Hiptmair2003} ensures that there exists $h_{\#}>0$ such that for all $h\in \paren{0, h_{\#}}$, we have
\begin{equation}
    \label{eq: discrete old inf-sup}
    \inf_{\paren{\EEs_h, \bm{\uplambda}_h} \in \XXs_h\times \YYs_h \setminus \left\{\bm{0}, \bm{0}\right\}} \sup_{\paren{\WWs_h, \bm{\varphi}_h} \in \XXs_h\times \YYs_h \setminus \left\{\bm{0}, \bm{0}\right\}} \dfrac{\abs{\bm{\mathcal{B}}_{\mathrm{C}}\paren{\paren{\EEs_h, \bm{\uplambda}_h}, \paren{\WWs_h, \bm{\varphi}_h}}}}{\norm{\paren{\EEs_h, \bm{\uplambda}_h}}_{\Om_{\text{s}}, \Gm} \norm{\paren{\WWs_h, \bm{\varphi}_h}}_{\Om_{\text{s}}, \Gm}} \ge C.
\end{equation}
In particular, Problem \eqref{eq: discrete old formulation} is uniquely solvable for all $h\in \paren{0, h_{\#}}$. 

To define the Galerkin approximation of Problem \eqref{eq: new formulation}, we next introduce the discrete counterparts of the operators used in \Cref{subsec: problem reformulation}. For each $\bm{\zeta}_h\in \YYs_h$, we define $\mathcal{E}_h\, \bm{\zeta}_h\in \XXs_h$ by
\begin{equation}
    \label{eq: discrete trace lifting operator}
    \left\{\begin{alignedat}{2}
    &a\paren{\mathcal{E}_h\, \bm{\zeta}_h, \vb_h}\  &=&\ \bm{0} \q \forall \vb_h \in \XXs_{h,0} := \XXs_h \cap \HHs_0\paren{\curlt, \Om_{\text{s}}}, \, \\
    &\gm_D^{\text{s}} \paren{\mathcal{E}_h\, \bm{\zeta}_h}\  &=&\ \bm{\zeta}_h.
    \end{alignedat}
    \right.
\end{equation}
Let $\mathcal{R}_h: \YYs_h \to \XXs_h$ be a discrete right inverse of $\gm_D^{\text{s}}$. For each $h\in \paren{0, h_{\#}}$, the operator $\mathcal{R}_h$ is bounded, but not necessarily uniformly. In what follows, no uniform boundedness is assumed unless explicitly stated otherwise. For the existence and construction of uniformly bounded discrete right inverses of $\gm_D^{\text{s}}$, we refer to \cite[Thm.~2.1]{DVS2003}, \cite[Thm.~5.6]{HM2012} or \cite[Thm.~2.3]{AGS2016}. We then define the discrete Steklov--Poincar{\' e} operator $\mathcal{S}_h: \YYs_h \to \YYs_h^\prime$ by
\begin{equation}
    \label{eq: discrete Steklov--Poincare operator}
    \inprod{\mathcal{S}_h\, \bm{\zeta}_h}{\overline{\bm{\eta}_h}}_{\YYs_h^\prime,\, \YYs_h} := a\paren{\mathcal{E}_h\, \bm{\zeta}_h, \mathcal{R}_h\, \bm{\eta}_h} \qqq \forall \bm{\eta}_h \in \YYs_h.
\end{equation}

\begin{lemma}
    \label{lem: properties of E_h and S_h}
    Assume that \Cref{assum: resonant frequency}(\ref{assum: resonant frequency:var}) is fulfilled. Then there exists $h_{\#}>0$ such that for all $h\in \paren{0, h_{\#}}$, the following results hold:
    \begin{enumerate}[label=(\roman*), ref=\roman*]
        \item Problem \eqref{eq: discrete trace lifting operator} is well-posed. Moreover, the operator $\mathcal{E}_h$ is uniformly bounded, i.e.,
        $$
        \norm{\mathcal{E}_h\, \bm{\zeta}_h}_{\HHs\paren{\curlt, \Om_{\text{s}}}}\le C \norm{\bm{\zeta}_h}_{\HHs^{-1/2}_{\times}\paren{\operatorname{div}_{\Gm}, \Gm}} \qqq \forall \bm{\zeta}_h\in \YYs_h.
        $$
        \item The definition \eqref{eq: discrete Steklov--Poincare operator} is independent of the choice of $\mathcal{R}_h$.
    \end{enumerate}
\end{lemma}

\begin{proof}
    Take a uniformly bounded right inverse $\mathcal{R}_h$ of $\gm_D^{\text{s}}$. By writing $\mathcal{E}_h\, \bm{\zeta}_h = \mathcal{R}_h\, \bm{\zeta}_h + \UUs_{\bm{\zeta}, h}$ with $\UUs_{\bm{\zeta}, h}\in \XXs_{h,0}$, Problem \eqref{eq: discrete trace lifting operator} is equivalent to 
    $$
    a\paren{\UUs_{\bm{\zeta}, h}, \vb_h} = \ell_{\bm{\zeta},h}\paren{\vb_h} \qqq \forall \vb_h \in \XXs_{h,0},
    $$
    where $\ell_{\bm{\zeta},h}\paren{\vb_h} = -a\paren{\mathcal{R}_h\, \bm{\zeta}_h, \vb_h}$. Under \Cref{assum: resonant frequency}(\ref{assum: resonant frequency:var}), this problem is well-posed \cite[Thm.~5.7]{Hiptmair2002}. This, together with the continuity of $a$ and the uniform boundedness of $\mathcal{R}_h$, implies that
    $$
    \norm{\UUs_{\bm{\zeta}, h}}_{\HHs\paren{\curlt, \Om_{\text{s}}}}\le C\norm{\ell_{\bm{\zeta},h}}_{\XXs_{h,0}^\prime} \le C \norm{\mathcal{R}_h\, \bm{\zeta}_h}_{\HHs\paren{\curlt, \Om_{\text{s}}}} \le C \norm{\bm{\zeta}_h}_{\HHs^{-1/2}_{\times}\paren{\operatorname{div}_{\Gm}, \Gm}},
    $$
    and consequently, $\mathcal{E}_h$ is uniformly bounded.

    The second part does not require $\mathcal{R}_h$ to be uniformly bounded. Let $\mathcal{R}_{1,h}$ and $\mathcal{R}_{2,h}$ denote two different discrete right inverses of $\gm_D^{\text{s}}$. Since $\mathcal{R}_{1,h}\, \bm{\eta}_h - \mathcal{R}_{2,h}\, \bm{\eta}_h \in \XXs_{h,0}$ for all $\bm{\eta}_h\in \YYs_h$, \eqref{eq: discrete trace lifting operator} gives us $a\paren{\mathcal{E}_h\, \bm{\zeta}_h, \mathcal{R}_{1,h}\, \bm{\eta}_h - \mathcal{R}_{2,h}\, \bm{\eta}_h} = 0$. Therefore, $a\paren{\mathcal{E}_h\, \bm{\zeta}_h, \mathcal{R}_{1,h}\, \bm{\eta}_h} = a\paren{\mathcal{E}_h\, \bm{\zeta}_h, \mathcal{R}_{2,h}\, \bm{\eta}_h}$, meaning that the definition of $\mathcal{S}_h$ does not depend on the choice of $\mathcal{R}_h$.
\end{proof}

With these discrete operators, the Galerkin approximation of Problem \eqref{eq: new formulation} reads as: Find $\paren{\bm{\tau}_h, \bm{\uplambda}_h}\in \YYs_h\times \YYs_h$ that satisfies
\begin{equation}
    \label{eq: discrete new formulation}
    \left\{\begin{alignedat}{3}
         &\inprod{\mathcal{S}_h\, \bm{\tau}_h}{\overline{\bm{\xi}_h}}_{\YYs_h^\prime,\, \YYs_h} + \om^2\varepsilon_0\inprod{\operatorname{V}_\kappa\bm{\tau}_h}{\overline{\bm{\xi}_h}}_{\times, \Gm} \ &+&\  \inprod{\paren{-\tfrac{1}{2}\operatorname{Id} + \operatorname{K}_\kappa}\bm{\uplambda}_h}{\overline{\bm{\xi}_h}}_{\times, \Gm} \ &=&\  \ell_{1,\Gm}\paren{\bm{\xi}_h}, \\
         &\inprod{\paren{\tfrac{1}{2}\operatorname{Id} + \operatorname{K}_\kappa}\bm{\tau}_h}{\overline{\bm{\varphi}_h}}_{\times, \Gm} \ &+&\ \mu_0\inprod{\operatorname{V}_\kappa \bm{\uplambda}_h}{\overline{\bm{\varphi}_h}}_{\times, \Gm} \ &=&\ \ell_2\paren{\bm{\vphi}_h},
    \end{alignedat}\right.
\end{equation}
for all $\paren{\bm{\xi}_h, \bm{\varphi_h}}\in \YYs_h\times \YYs_h$. We also need the space
$$
\XXs_h^{\mathcal{E}} := \left\{\ub_h \in \XXs_h \mid a\paren{\ub_h, \vb_h} = \bm{0} \text{ for all } \vb_h \in \XXs_{h,0}\right\}.
$$
The following result is analogous to \Cref{theo: continuous problems equivalence}:

\begin{theorem}
    \label{theo: discrete problems equivalence}
    Assume that \Cref{assum: resonant frequency} holds. Then there exists $h_{\#}>0$ such that for all $h\in \paren{0, h_{\#}}$, Problems \eqref{eq: discrete old formulation} and \eqref{eq: discrete new formulation} are equivalent.
\end{theorem}

\begin{proof}
    Choose $h_{\#}>0$ sufficiently small so that Problems \eqref{eq: discrete old formulation} and \eqref{eq: discrete trace lifting operator} are well-posed for all $h\in \paren{0, h_{\#}}$. Denote by $\paren{\EEs_h, \bm{\uplambda}_h}\in \XXs_h\times \YYs_h$ the solution of Problem \eqref{eq: discrete old formulation}, and let $\bm{\tau}_h = \gm_D^{\text{s}} \EEs_h$. Taking $\WWs_h \in \XXs_{h,0}$ in \eqref{eq: discrete old formulation} and using \eqref{eq: discrete trace lifting operator}, we obtain $\EEs_h = \mathcal{E}_h\, \bm{\tau}_h$, and hence $\EEs_h\in \XXs_h^{\mathcal{E}}$. Therefore, the trial space of Problem \eqref{eq: discrete old formulation} may be restricted to $\XXs_h^{\mathcal{E}}\times \YYs_h$. Next, consider an arbitrary discrete right inverse $\mathcal{R}_h$ of $\gm_D^{\text{s}}$. Let $\WWs_h\in \XXs_h$ be arbitrary, and denote $\bm{\xi}_h = \gm_D^{\text{s}} \WWs_h$.  Then $\WWs_h$ admits the decomposition $\WWs_h = \mathcal{R}_h\, \bm{\xi}_h + \paren{\WWs_h - \mathcal{R}_h\, \bm{\xi}_h}$, where $\mathcal{R}_h\, \bm{\xi}_h \in \mathcal{R}_h\paren{\YYs_h}$ and $ \WWs_h - \mathcal{R}_h\, \bm{\xi}_h \in \XXs_{h,0}$. Together with \eqref{eq: discrete trace lifting operator}, we see that it is sufficient to test Problem \eqref{eq: discrete old formulation} with pairs in $\mathcal{R}_h \paren{\YYs_h} \times \YYs_h$. Hence, Problem \eqref{eq: discrete old formulation} is equivalent to the following reduced formulation: Find $\paren{\mathcal{E}_h\, \bm{\tau}_h, \bm{\uplambda}_h}\in \XXs_h^{\mathcal{E}}\times \YYs_h$ such that
    $$
    \left\{\begin{alignedat}{3}
        &a\paren{\mathcal{E}_h\, \bm{\tau}_h, \mathcal{R}_h\, \bm{\xi}_h} + \om^2\varepsilon_0\inprod{\operatorname{V}_{\kappa}\paren{\gm_D^{\text{s}}\, \mathcal{E}_h\, \bm{\tau}_h}}{\gm_D^{\text{s}}\,  \overline{\mathcal{R}_h\, \bm{\xi}_h}}_{\times, \Gm} \ &+&\  \inprod{\paren{-\tfrac{1}{2}\operatorname{Id} + \operatorname{K}_\kappa}\bm{\uplambda}_h}{\gm_D^{\text{s}}\, \overline{\mathcal{R}_h\, \bm{\xi}_h}}_{\times, \Gm} \ &=&\  \ell_{1,\Gm}\paren{\bm{\xi}_h}, \\
        &\inprod{\paren{\tfrac{1}{2}\operatorname{Id} + \operatorname{K}_\kappa}\paren{\gm_D^{\text{s}}\, \mathcal{E}_h\, \bm{\tau}_h}}{\overline{\bm{\varphi}_h}}_{\times, \Gm} \ &+&\ \mu_0\inprod{\operatorname{V}_\kappa \bm{\uplambda}_h}{\overline{\bm{\varphi}_h}}_{\times, \Gm} \ &=&\ \ell_2\paren{\bm{\vphi}_h},
    \end{alignedat}\right.
    $$
    for all $\paren{\mathcal{R}_h\, \bm{\xi}_h, \bm{\varphi}_h}\in \mathcal{R}_h\paren{\YYs_h}\times \YYs_h$. Using \eqref{eq: discrete Steklov--Poincare operator}, this becomes Problem \eqref{eq: discrete new formulation}. Conversely, let $\paren{\bm{\tau}_h, \bm{\uplambda}_h}\in \YYs_h\times \YYs_h$ be the solution of Problem \eqref{eq: discrete new formulation}. Then by \eqref{eq: discrete trace lifting operator} and \eqref{eq: discrete Steklov--Poincare operator}, the pair $\paren{\mathcal{E}_h\, \bm{\tau}_h, \bm{\uplambda}_h}$ satisfies the reduced formulation above, and therefore also Problem \eqref{eq: discrete old formulation}. 
\end{proof}

Let $\bm{\mathcal{B}}_{\text{R}} : \paren{\YYs_h\times \YYs_h}\times \paren{\YYs_h\times \YYs_h}\to \C$ denote the sesquilinear form that corresponds to the left-hand side of \eqref{eq: discrete new formulation}. In view of \Cref{theo: discrete problems equivalence}, the inf-sup condition for $\bm{\mathcal{B}}_{\text{R}}$ follows directly from \eqref{eq: discrete old inf-sup}. Therefore, it remains to show that the associated inf-sup constant for $\bm{\mathcal{B}}_{\text{R}}$ is independent of $h$. To this end, we  recall the canonical product norm on the space $\HHs^{-1/2}_{\times}\paren{\operatorname{div}_{\Gm}, \Gm}\times \HHs^{-1/2}_{\times}\paren{\operatorname{div}_{\Gm}, \Gm}$:
$$
\norm{\paren{\bm{\zeta}, \bm{\eta}}}_{\Gm, \Gm}^2 := \norm{\bm{\zeta}}^2_{\HHs^{-1/2}_{\times}\paren{\operatorname{div}_{\Gm}, \Gm}} + \norm{\bm{\eta}}^2_{\HHs^{-1/2}_{\times}\paren{\operatorname{div}_{\Gm}, \Gm}} \qqq\forall \paren{\bm{\zeta}, \bm{\eta}}\in \HHs^{-1/2}_{\times}\paren{\operatorname{div}_{\Gm}, \Gm}\times \HHs^{-1/2}_{\times}\paren{\operatorname{div}_{\Gm}, \Gm}.
$$

\begin{theorem}
    \label{theo: discrete inf-sup condition}
    Under \Cref{assum: resonant frequency}, there exists a constant $C>0$ such that
    $$
    \sup_{\paren{\bm{\xi}_h, \bm{\varphi}_h} \in \YYs_h\times \YYs_h \setminus \left\{\bm{0}, \bm{0}\right\}} \dfrac{\abs{\bm{\mathcal{B}}_{\mathrm{R}}\paren{\paren{\bm{\tau}_h, \bm{\uplambda}_h}, \paren{\bm{\xi}_h, \bm{\varphi}_h}}}}{\norm{\paren{\bm{\xi}_h, \bm{\varphi}_h}}_{\Gm, \Gm}} \ge C\ \norm{\paren{\bm{\tau}_h, \bm{\uplambda}_h}}_{\Gm, \Gm} \qqq \forall \paren{\bm{\tau}_h, \bm{\uplambda}_h} \in \YYs_h\times \YYs_h,
    $$
    for all $h\in \paren{0, h_{\#}}$, where $h_{\#}>0$ is sufficiently small.
\end{theorem}

\begin{proof}
    By \eqref{eq: discrete old inf-sup} and the inclusion $\XXs_h^{\mathcal{E}}\subset\XXs_h$, there exists $h_{\#}>0$ such that for all $h\in \paren{0, h_{\#}}$, the following inequality holds:
    \begin{equation}
        \label{eq: discrete inf-sup condition 1}
        C \le \inf_{\paren{\EEs_h, \bm{\uplambda}_h}\in \XXs_h^{\mathcal{E}}\times \YYs_h \setminus\left\{\bm{0}, \bm{0}\right\}} \sup_{\paren{\WWs_h, \bm{\varphi}_h}\in \XXs_h\times \YYs_h \setminus\left\{\bm{0}, \bm{0}\right\}} \dfrac{\abs{\bm{\mathcal{B}}_{\text{C}}\paren{\paren{\EEs_h, \bm{\uplambda}_h}, \paren{\WWs_h, \bm{\varphi}_h}}}}{\norm{\paren{\EEs_h, \bm{\uplambda}_h}}_{\Om_{\text{s}}, \Gm} \norm{\paren{\WWs_h, \bm{\varphi}_h}}_{\Om_{\text{s}}, \Gm}}.
    \end{equation}
    From the proof of \Cref{theo: discrete problems equivalence}, every $\EEs_h \in \XXs_h^{\mathcal{E}}$ can be written as $\EEs_h =\mathcal{E}_h\, \bm{\tau}_h$, where $\bm{\tau}_h = \gm_D^{\text{s}} \EEs_h$. Consider a uniformly bounded discrete right inverse $\mathcal{R}_h$ of $\gm_D^{\text{s}}$. Then for any $\WWs_h\in \XXs_h$, if $\bm{\xi}_h = \gm_D^{\text{s}}\WWs_h$, then $\WWs_h - \mathcal{R}_h\, \bm{\xi}_h\in \XXs_{h,0}$. By using \eqref{eq: discrete trace lifting operator}, we hence obtain
    $$
    a\paren{\EEs_h, \WWs_h - \mathcal{R}_h\, \bm{\xi}_h} = a\paren{\mathcal{E}_h\, \bm{\tau}_h, \WWs_h - \mathcal{R}_h\, \bm{\xi}_h} = 0,
    $$
    which leads to
    $$
    \bm{\mathcal{B}}_{\text{C}}\paren{\paren{\EEs_h, \bm{\uplambda}_h}, \paren{\WWs_h, \bm{\varphi}_h}}=\bm{\mathcal{B}}_{\text{C}}\paren{\paren{\mathcal{E}_h\, \bm{\tau}_h, \bm{\uplambda}_h}, \paren{\mathcal{R}_h\, \bm{\xi}_h, \bm{\varphi}_h}}.
    $$
    By using the uniform boundedness of $\mathcal{R}_h$ and the continuity of $\gm_D^{\text{s}}$, we have
    $$
    \norm{\paren{\mathcal{R}_h\, \bm{\xi}_h, \bm{\varphi}_h}}_{\Om_{\text{s}}, \Gm} \le C \norm{\paren{\bm{\xi}_h, \bm{\vphi}_h}}_{\Gm, \Gm} \le C\norm{\paren{\WWs_h, \bm{\varphi}_h}}_{\Om_{\text{s}}, \Gm} \qqq \forall \bm{\varphi}_h\in \YYs_h.
    $$
    It follows that
    \begin{equation}
        \label{eq: discrete inf-sup condition 2}
        \sup_{\paren{\WWs_h, \bm{\varphi}_h}\in \XXs_h\times \YYs_h \setminus\left\{\bm{0}, \bm{0}\right\}} \dfrac{\abs{\bm{\mathcal{B}}_{\text{C}}\paren{\paren{\EEs_h, \bm{\uplambda}_h}, \paren{\WWs_h, \bm{\varphi}_h}}}}{\norm{\paren{\EEs_h, \bm{\uplambda}_h}}_{\Om_{\text{s}}, \Gm} \norm{\paren{\WWs_h, \bm{\varphi}_h}}_{\Om_{\text{s}}, \Gm}} \le C \sup_{\paren{\mathcal{R}_h\, \bm{\xi}_h, \bm{\varphi}_h}\in \mathcal{R}_h\paren{\YYs_h}\times \YYs_h \setminus\left\{\bm{0}, \bm{0}\right\}} \dfrac{\abs{\bm{\mathcal{B}}_{\text{C}}\paren{\paren{\mathcal{E}_h\, \bm{\tau}_h, \bm{\uplambda}_h}, \paren{\mathcal{R}_h\, \bm{\xi}_h, \bm{\varphi}_h}}}}{\norm{\paren{\mathcal{E}_h\, \bm{\tau}_h, \bm{\uplambda}_h}}_{\Om_{\text{s}}, \Gm} \norm{\paren{\mathcal{R}_h\, \bm{\xi}_h, \bm{\varphi}_h}}_{\Om_{\text{s}}, \Gm}},
    \end{equation}
    for all $\paren{\EEs_h, \bm{\uplambda}_h}\in \XXs_h^{\mathcal{E}}\times \YYs_h \setminus \left\{\bm{0}, \bm{0}\right\}$. From the proof of \Cref{theo: discrete problems equivalence}, we have
    $$
    \bm{\mathcal{B}}_{\text{C}}\paren{\paren{\mathcal{E}_h\, \bm{\tau}_h, \bm{\uplambda}_h}, \paren{\mathcal{R}_h\, \bm{\xi}_h, \bm{\varphi}_h}} = \bm{\mathcal{B}}_{\text{R}}\paren{\paren{\bm{\tau}_h, \bm{\uplambda}_h}, \paren{\bm{\xi}_h, \bm{\varphi}_h}} \qqq \forall \paren{\bm{\tau}_h, \bm{\uplambda}_h}, \paren{\bm{\xi}_h, \bm{\varphi}_h} \in \YYs_h\times \YYs_h.
    $$
    Moreover, the maps $\paren{\bm{\tau}_h, \bm{\uplambda}_h} \mapsto \paren{\mathcal{E}_h\, \bm{\tau}_h, \bm{\uplambda}_h}$ and $\paren{\bm{\xi}_h, \bm{\vphi}_h} \mapsto \paren{\mathcal{R}_h\, \bm{\xi}_h, \bm{\varphi}_h}$ from $\YYs_h\times \YYs_h$ into $\XXs_h^{\mathcal{E}}\times \YYs_h$ and $\mathcal{R}_h\paren{\YYs_h}\times \YYs_h$, respectively, are bijective. Together with the continuity of $\gm_D^{\text{s}}$, we finally obtain
    \begin{align*}
        &\inf_{\paren{\mathcal{E}_h\, \bm{\tau}_h, \bm{\uplambda}_h}\in \XXs_h^{\mathcal{E}}\times \YYs_h \setminus\left\{\bm{0}, \bm{0}\right\}} \sup_{\paren{\mathcal{R}_h\, \bm{\xi}_h, \bm{\varphi}_h}\in \mathcal{R}_h\paren{\YYs_h}\times \YYs_h \setminus\left\{\bm{0}, \bm{0}\right\}} \dfrac{\abs{\bm{\mathcal{B}}_{\text{C}}\paren{\paren{\mathcal{E}_h\, \bm{\tau}_h, \bm{\uplambda}_h}, \paren{\mathcal{R}_h\, \bm{\xi}_h, \bm{\varphi}_h}}}}{\norm{\paren{\mathcal{E}_h\, \bm{\tau}_h, \bm{\uplambda}_h}}_{\Om_{\text{s}}, \Gm} \norm{\paren{\mathcal{R}_h\, \bm{\xi}_h, \bm{\varphi}_h}}_{\Om_{\text{s}}, \Gm}} \\
        &\qq = \inf_{\paren{\bm{\tau}_h, \bm{\uplambda}_h}\in \YYs_h\times \YYs_h \setminus\left\{\bm{0}, \bm{0}\right\}} \sup_{\paren{\bm{\xi}_h, \bm{\varphi}_h}\in \YYs_h\times \YYs_h \setminus\left\{\bm{0}, \bm{0}\right\}} \dfrac{\abs{\bm{\mathcal{B}}_{\text{R}}\paren{\paren{\bm{\tau}_h, \bm{\uplambda}_h}, \paren{\bm{\xi}_h, \bm{\varphi}_h}}}}{\norm{\paren{\mathcal{E}_h\, \bm{\tau}_h, \bm{\uplambda}_h}}_{\Om_{\text{s}}, \Gm} \norm{\paren{\mathcal{R}_h\, \bm{\xi}_h, \bm{\varphi}_h}}_{\Om_{\text{s}}, \Gm}} \\
        &\qq \le C \inf_{\paren{\bm{\tau}_h, \bm{\uplambda}_h}\in \YYs_h\times \YYs_h \setminus\left\{\bm{0}, \bm{0}\right\}} \sup_{\paren{\bm{\xi}_h, \bm{\varphi}_h}\in \YYs_h\times \YYs_h \setminus\left\{\bm{0}, \bm{0}\right\}} \dfrac{\abs{\bm{\mathcal{B}}_{\text{R}}\paren{\paren{\bm{\tau}_h, \bm{\uplambda}_h}, \paren{\bm{\xi}_h, \bm{\varphi}_h}}}}{\norm{\paren{\bm{\tau}_h, \bm{\uplambda}_h}}_{\Gm, \Gm} \norm{\paren{\bm{\xi}_h, \bm{\varphi}_h}}_{\Gm, \Gm}}.
    \end{align*}
    By combining \eqref{eq: discrete inf-sup condition 1}, the inequality obtained by taking the infimum over $\paren{\EEs_h, \bm{\uplambda}_h}\in \XXs_h^{\mathcal{E}}\times \YYs_h \setminus \left\{\bm{0}, \bm{0}\right\}$ in \eqref{eq: discrete inf-sup condition 2}, and the last estimate, we finish the proof.
\end{proof}

\section{Operator preconditioning}
\label{sec: operator preconditioning}

This section is devoted to the matrix representation of Problem \eqref{eq: discrete old formulation}. We then introduce a preconditioner for the resulting discrete system following the operator preconditioning framework \cite[Sect.~2]{Hiptmair2006}. For fixed $h\in \paren{0, h_{\#}}$, let $\left\{\bm{\psi}_1, \bm{\psi}_2, \ldots, \bm{\psi}_N\right\}$ be a basis of $\XXs_h$, where $N := \dim \XXs_h$. Since each basis function is associated with a unique edge of the mesh, we split the index set into
$$
\mathcal{J}_{\mathbf{I}}:= \left\{i \mid e_i \not\subset \Gm\right\}, \qqq \mathcal{J}_{\mathbf{B}}:= \left\{i\mid e_i \subset \Gm\right\}.
$$
With respect to this decomposition, the Galerkin matrix associated with $a\paren{\cdot, \cdot}$ can be written as
$$
\mathbf{A} = \begin{pmatrix}
    \mathbf{A}_{\mathbf{II}} & \mathbf{A}_{\mathbf{IB}} \\
    \mathbf{A}_{\mathbf{BI}} & \mathbf{A}_{\mathbf{BB}}
\end{pmatrix},
$$
where
$$
\left[\mathbf{A}_{\sigma \varsigma}\right]_{kl}:= a\paren{\bm{\psi}_l, \bm{\psi}_k}, \qqq \sigma, \varsigma \in \left\{\mathbf{I}, \mathbf{B}\right\},\ k \in \mathcal{J}_{\mathcal{\sigma}}, l \in \mathcal{J}_{\mathcal{\varsigma}}.
$$

By construction, we have $\XXs_{h,0} = \operatorname{span}\left\{\bm{\psi}_i \mid i\in \mathcal{J}_{\mathbf{I}}\right\}$ and $\XXs_h = \XXs_{h,0}\oplus \operatorname{span}\left\{\bm{\psi}_i \mid i\in \mathcal{J}_{\mathbf{B}}\right\}$. Moreover, since $\YYs_h = \gm_D^{\text{s}} \XXs_h$, the set $\left\{\gm_D^{\text{s}}\, \bm{\psi}_i \mid i\in \mathcal{J}_{\mathbf{B}}\right\}$ forms a basis of $\YYs_h$. Hence, for $\bm{\zeta}_h \in \YYs_h$ with coefficient vector $\widehat{\bm{\zeta}} \in \C^{\, \abs{\mathcal{J}_{\mathbf{B}}}}$, the solution $\mathcal{E}_h\, \bm{\zeta}_h\in \XXs_h$ of Problem \eqref{eq: discrete trace lifting operator} has the coefficient vector $\paren{\widehat{\UUs}_{\bm{\zeta}},\ \widehat{\bm{\zeta}}}^\top \in \C^N$, where $\widehat{\UUs}_{\bm{\zeta}}\in \C^{\, \abs{\mathcal{J}_{\mathbf{I}}}}$ collects the coefficients of all $\bm{\psi}_i\ \paren{i \in \mathcal{J}_{\mathbf{I}}}$. In other words, we have
\begin{equation}
    \label{eq: coordinate representation}
    \mathcal{E}_h\, \bm{\zeta}_h = \sum_{i\in \mathcal{J}_{\mathbf{I}}} \mathrm{U}_i\, \bm{\psi}_i + \sum_{i\in \mathcal{J}_{\mathbf{B}}} \zeta_i\,  \bm{\psi}_i, \qqq \widehat{\UUs}_{\bm{\zeta}} = \paren{\mathrm{U}_i}_{i\in \mathcal{J}_{\mathbf{I}}},\  \widehat{\bm{\zeta}} = \paren{\zeta_i}_{i\in \mathcal{J}_{\mathbf{B}}}.
\end{equation}
By substituting this back into \eqref{eq: discrete trace lifting operator}, we hence obtain
\begin{equation}
    \label{eq: matrix trace lifting equation}
    \mathbf{A}_{\mathbf{II}}\, \widehat{\UUs}_{\bm{\zeta}} + \mathbf{A}_{\mathbf{IB}}\, \widehat{\bm{\zeta}} = \bm{0}.
\end{equation}
Assume that \Cref{assum: resonant frequency}(\ref{assum: resonant frequency:var}) is satisfied. The well-posedness of Problem \eqref{eq: discrete trace lifting operator} implies that $A_{\mathbf{II}}$ is invertible, and hence
$$
\widehat{\UUs}_{\bm{\zeta}} = - \mathbf{A}_{\mathbf{II}}^{-1}\, \mathbf{A}_{\mathbf{IB}}\, \widehat{\bm{\zeta}}.
$$
Now let $\bm{\eta}_h\in \YYs_h$ have coefficient vector $\widehat{\bm{\eta}}\in \C^{\, \abs{\mathcal{J}_{\mathbf{B}}}}$. Since the definition of $\mathcal{S}_h$ does not depend on the choice of $\mathcal{R}_h$, we may take $\mathcal{R}_h = \mathcal{E}_h$ in \eqref{eq: discrete Steklov--Poincare operator}. Using the above expression for $\widehat{\UUs}_{\bm{\zeta}}$, we deduce that
$$
\inprod{\mathcal{S}_h\, \bm{\zeta}_h}{\overline{\bm{\eta}_h}}_{\YYs_h^\prime,\, \YYs_h} = a\paren{\mathcal{E}_h\, \bm{\zeta}_h, \mathcal{E}_h\, \bm{\eta}_h} = \widehat{\bm{\eta}}^{\, \ast}\paren{\mathbf{A}_{\mathbf{BB}} - \mathbf{A}_{\mathbf{BI}}\, \mathbf{A}_{\mathbf{II}}^{-1}\, \mathbf{A}_{\mathbf{IB}}}\widehat{\bm{\zeta}}.
$$
As a result, the Galerkin matrix of the sesquilinear form induced by $\mathcal{S}_h$ is the Schur complement of $\mathbf{A}_{\mathbf{II}}$ in $\mathbf{A}$:
$$
\mathbf{S}_{\text{sc}}:= \mathbf{A}_{\mathbf{BB}} - \mathbf{A}_{\mathbf{BI}}\, \mathbf{A}_{\mathbf{II}}^{-1}\, \mathbf{A}_{\mathbf{IB}}.
$$
We next define the Galerkin matrices
$$
\left[\mathbf{K}_{\pm}\right]_{kl} := \inprod{\paren{\pm \tfrac{1}{2}\operatorname{Id} + \operatorname{K}_\kappa}\gm_D^{\text{s}}\, \bm{\psi}_l}{\gm_D^{\text{s}}\, \overline{\bm{\psi}_k}}_{\times, \Gm}, \qqq \left[\mathbf{V}\right]_{kl} := \inprod{\operatorname{V}_\kappa \gm_D^{\text{s}}\, \bm{\psi}_l}{\gm_D^{\text{s}}\, \overline{\bm{\psi}_k}}_{\times, \Gm},
$$
and the load vectors
\begin{align*}
    \left[\mathbf{l}_1\right]_k &:= \om^2\varepsilon_0\inprod{\operatorname{V}_{\kappa}\paren{\gm_D^0 \EEs^{\text{inc}}}}{\gm_D^{\text{s}}\, \overline{\bm{\psi}_k}}_{\times, \Gm} + \inprod{\gm_N^0 \EEs^{\text{inc}}}{\gm_D^{\text{s}}\, \overline{\bm{\psi}_k}}_{\times, \Gm}, \\
    \left[\mathbf{l}_2\right]_k &:= \inprod{\paren{\tfrac{1}{2}\operatorname{Id} + \operatorname{K}_\kappa}\paren{\gm_D^0 \EEs^{\text{inc}}}}{\gm_D^{\text{s}}\, \overline{\bm{\psi}_k}}_{\times, \Gm},
\end{align*}
for $k, l \in \left\{1,2,\ldots, \abs{\mathcal{J}_{\mathbf{B}}}\right\}$. Let $\paren{\widehat{\bm{\tau}},\ \widehat{\bm{\uplambda}}}^\top \in \C^{\, 2\abs{\mathcal{J}_{\mathbf{B}}}}$ denote the coefficient vector of the solution $\paren{\bm{\tau}_h, \bm{\uplambda}_h}\in \YYs_h\times \YYs_h$ of Problem \eqref{eq: discrete new formulation}. Then the resulting discrete system reads as
\begin{equation}
    \label{eq: matrix equation}
    \mathbb{B}_{\text{R}}\begin{pmatrix}
        \widehat{\bm{\tau}}\\
        \widehat{\bm{\uplambda}}
    \end{pmatrix} = \begin{pmatrix}
        \mathbf{l}_1\\
        \mathbf{l}_2
    \end{pmatrix}, \qqq \mathbb{B}_{\text{R}} := \begin{pmatrix}
        \mathbf{S}_{\text{sc}} + \om^2\varepsilon_0\, \mathbf{V} & \mathbf{K}_{-}\\
        \mathbf{K}_{+} & \mu_0\, \mathbf{V}
    \end{pmatrix}.
\end{equation}
Under \Cref{assum: resonant frequency}, there exists $h_{\#}>0$ such that for all $h\in \paren{0, h_{\#}}$, this system is uniquely solvable.

Now, recall from \cite[Sect.~2]{BC2007} the Buffa--Christiansen space $\ZZs_h\subset \HHs^{-1/2}\paren{\operatorname{div}_\Gm, \Gm}$ constructed on the barycentric refinement of $\Gm_h$. Note that $\operatorname{dim} \ZZs_h = \operatorname{dim} \YYs_h$ for all $h>0$. Denote by $\mathbb{V}$ and $\mathbb{D}$ the Galerkin matrices associated with the following sesquilinear forms:
\begin{equation}
    \label{eq: sesquilinear forms for preconditioners}
    \begin{alignedat}{2}
        &\bm{\mathcal{V}} : \paren{\ZZs_h\times \ZZs_h}\times \paren{\ZZs_h\times \ZZs_h}\to \C, &\qq &\bm{\mathcal{V}}\paren{\paren{\bm{\rho}_h, \bm{\vartheta}_h}, \paren{\bm{\iota}_h, \bm{\chi}_h}}:= \inprod{\operatorname{V}_\kappa \bm{\rho}_h}{\, \overline{\bm{\iota}_h}}_{\times, \Gm} + \inprod{\operatorname{V}_\kappa \bm{\vartheta}_h}{\, \overline{\bm{\chi}_h}}_{\times, \Gm},\\
        &\bm{\mathcal{D}} : \paren{\YYs_h\times \YYs_h}\times \paren{\ZZs_h\times \ZZs_h}\to \C, &\qq &\bm{\mathcal{D}}\paren{\paren{\bm{\zeta}_h, \bm{\eta}_h}, \paren{\bm{\iota}_h, \bm{\chi}_h}}:= \inprod{\bm{\zeta}_h}{\, \overline{\bm{\iota}_h}}_{\times, \Gm} + \inprod{\bm{\eta}_h}{\, \overline{\bm{\chi}_h}}_{\times, \Gm},
    \end{alignedat}
\end{equation}
we then propose a block-diagonal preconditioner $\mathbb{P}$ for the matrix $\mathbb{B}_{\text{R}}$ in \eqref{eq: matrix equation}:
$$
\mathbb{P}:= \mathbb{D}^{-1}\, \mathbb{V}\, \mathbb{D}^{-\top}.
$$
This choice is motivated by the regularization of the Steklov--Poincar{\' e} operator via the operator $\operatorname{V}_\kappa$ on smooth surfaces \cite[Thm.~2.2]{DZH+2015}, and by operator preconditioning for the discrete system arising from a Galerkin discretization of the electric field integral equation on Lipschitz surfaces \cite[Prop.~4.4]{BC2007}. 

For a square matrix $\mathbf{M}$, we define the spectral condition number $\operatorname{cond}\paren{\mathbf{M}}$ as the ratio between the maximum and minimum moduli of its eigenvalues. We have the following result:

\begin{theorem}
    \label{theo: operator precondition}
    Assume that \Cref{assum: resonant frequency} is satisfied. Then there exists $h_{\#}>0$ such that for all $h\in \paren{0, h_{\#}}$, the following inequality holds:
    $$
    \operatorname{cond}\paren{\mathbb{PB}_{\mathrm{R}}} \le C.
    $$
\end{theorem}

\begin{proof}
    From the proof of \cite[Thm.~15]{BH2003}, it follows directly that $\bm{\mathcal{V}}$ satisfies a uniform discrete inf-sup condition for sufficiently small $h \in \paren{0, h_{\#}}$. Using \cite[Prop.~3.14]{BC2007}, we obtain the uniform inf-sup stability of $\bm{\mathcal{D}}$ for all $h>0$. Combined with \Cref{theo: discrete inf-sup condition}, the result is now a consequence of \cite[Thm.~2.1]{Hiptmair2006}.
\end{proof}

\begin{remark}
    \label{rem: other preconditioners}
    As discussed above, the particular choice of $\bm{\mathcal{V}}$ and $\bm{\mathcal{D}}$ in \eqref{eq: sesquilinear forms for preconditioners} is motivated by \cite{DZH+2015, BC2007}. However, the proof of \Cref{theo: operator precondition} does not rely on the specific structure of these forms, but only on their uniform stability. Consequently, the same argument remains valid for any preconditioner defined through sesquilinear forms enjoying the same stability properties.
\end{remark}

\begin{remark}
    \label{rem: advantage of the method}
    The proof of \Cref{theo: discrete problems equivalence} shows that if $\paren{\bm{\tau}_h, \bm{\uplambda}_h}$ solves Problem \eqref{eq: discrete new formulation}, then $\paren{\mathcal{E}_h\, \bm{\tau}_h, \bm{\uplambda}_h}$ solves Problem \eqref{eq: discrete old formulation}. At the algebraic level, this corresponds to first solving \eqref{eq: matrix equation}, and then recovering $\mathcal{E}_h\, \bm{\tau}_h$ from \eqref{eq: coordinate representation}--\eqref{eq: matrix trace lifting equation}. The main computational burden lies in \eqref{eq: matrix equation}, since it contains dense BEM blocks, whereas \eqref{eq: matrix trace lifting equation} is a sparse FEM system. Available FEM solvers, including direct solvers \cite{DD1997, ADLK2001} and preconditioned iterative methods \cite{Hiptmair1999, GP2003, QDA+2026}, can be used both to apply the Schur complement $\mathbf{S}_{\mathrm{sc}}$ and to solve \eqref{eq: matrix trace lifting equation}. Therefore, preconditioning of \eqref{eq: matrix equation} provides a practical strategy to solve the original problem. 
\end{remark}

\section{Numerical results}
\label{sec: numerical results}

In this section, we present numerical experiments to illustrate the theoretical findings. In all examples, the exterior domain $\Om_0$ is filled with air, with permeability $\mu_0 = 4\pi\times 10^{-7}\ \mathrm{H}/\mathrm{m}$ and permittivity $\varepsilon_0 = 8.85\times 10^{-12}\ \mathrm{F}/\mathrm{m}$. We choose the angular frequency $\om = 6.28\times 10^8\ \mathrm{rad}/\mathrm{s}$. 

For a range of mesh sizes, we compute the condition numbers of the matrices $\mathbb{B}_{\text{C}}$, $\mathbb{B}_{\text{R}}$, and $\mathbb{PB}_{\mathrm{R}}$, where $\mathbb{B}_{\text{C}}$ denotes the Galerkin matrix of the discrete Costabel coupling \eqref{eq: discrete old formulation}. To complement the condition number study, we also report the corresponding GMRES iteration counts, where GMRES \cite{SS1986} is used with both absolute and relative tolerances set to $10^{-9}$. The comparison between $\mathbb{B}_{\text{C}}$ and $\mathbb{B}_{\text{R}}$ illustrates the effect of the reformulation, while that between $\mathbb{B}_{\text{R}}$ and $\mathbb{PB}_{\mathrm{R}}$ demonstrates the efficiency of the preconditioner. All experiments are carried out in BEAST.jl \cite{Kristof2026}. In particular, in the evaluation of $\mathbf{S}_{\text{sc}}$, the auxiliary system with coefficient matrix $\mathbf{A}_{\mathbf{II}}$ is solved using a sparse LU factorization \cite{DD1997}. 

\begin{figure}[!h]
    \centering
    \begin{subfigure}[c]{0.44\textwidth}
        \centering
        \raisebox{-0.5\height}{%
            \includegraphics[trim={2cm 2cm 2cm 2cm},clip,width=0.87\linewidth]{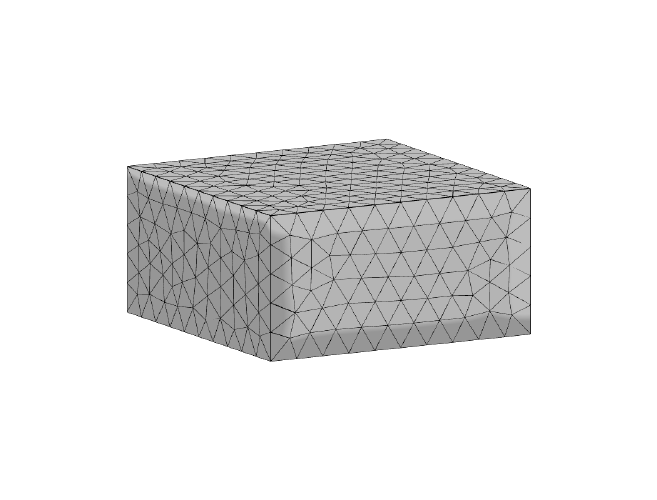}
        }
    \end{subfigure}
    \hfill
    \begin{subfigure}[c]{0.53\textwidth}
        \centering
        \raisebox{-0.5\height}{%
            \includegraphics[trim={34cm 29cm 34cm 17cm},clip,width=0.72\linewidth]{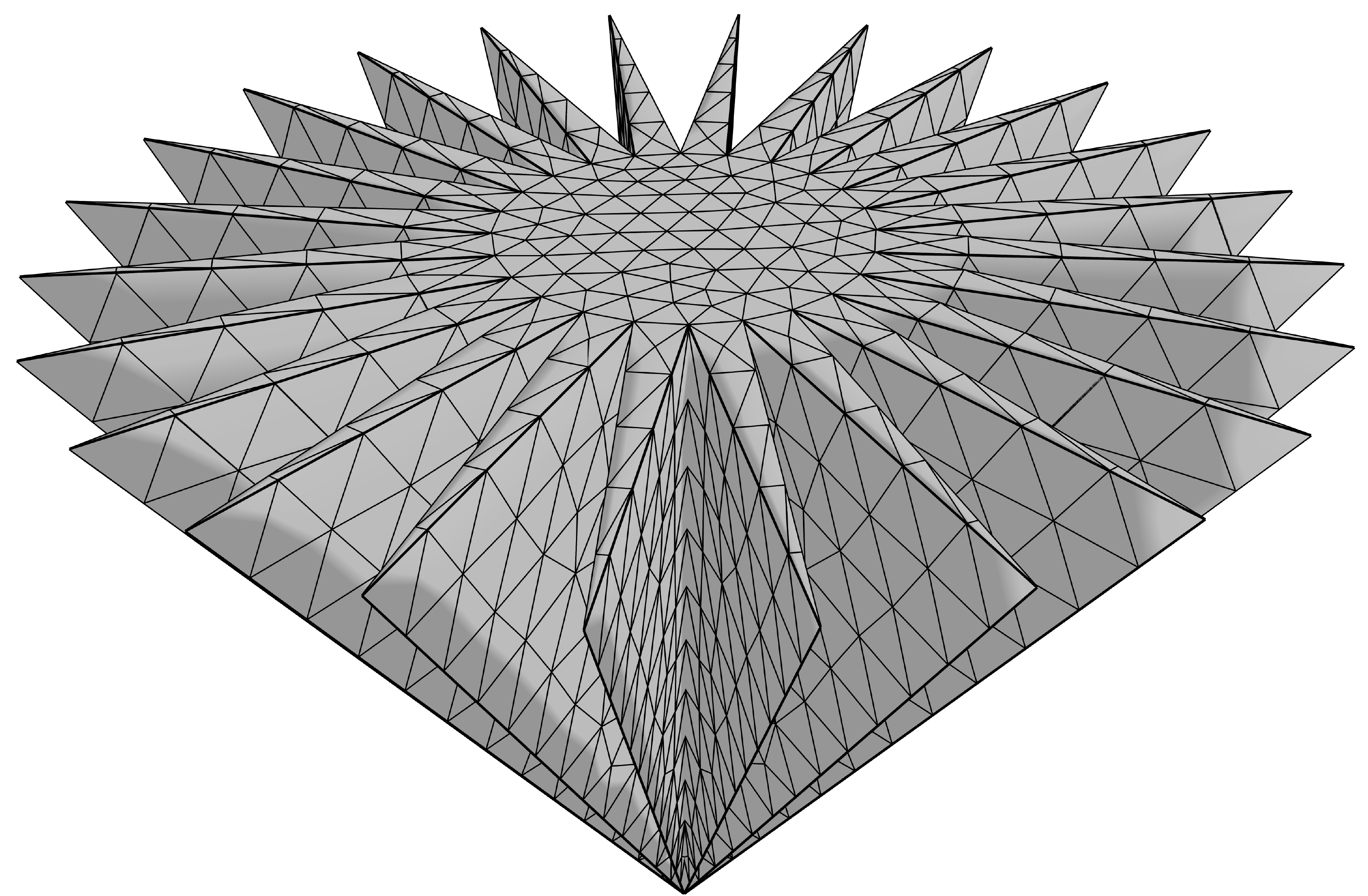}%
            \hspace{-0.35cm}%
            \includegraphics[trim={38cm 18cm 38cm 18cm},clip,width=0.18\linewidth]{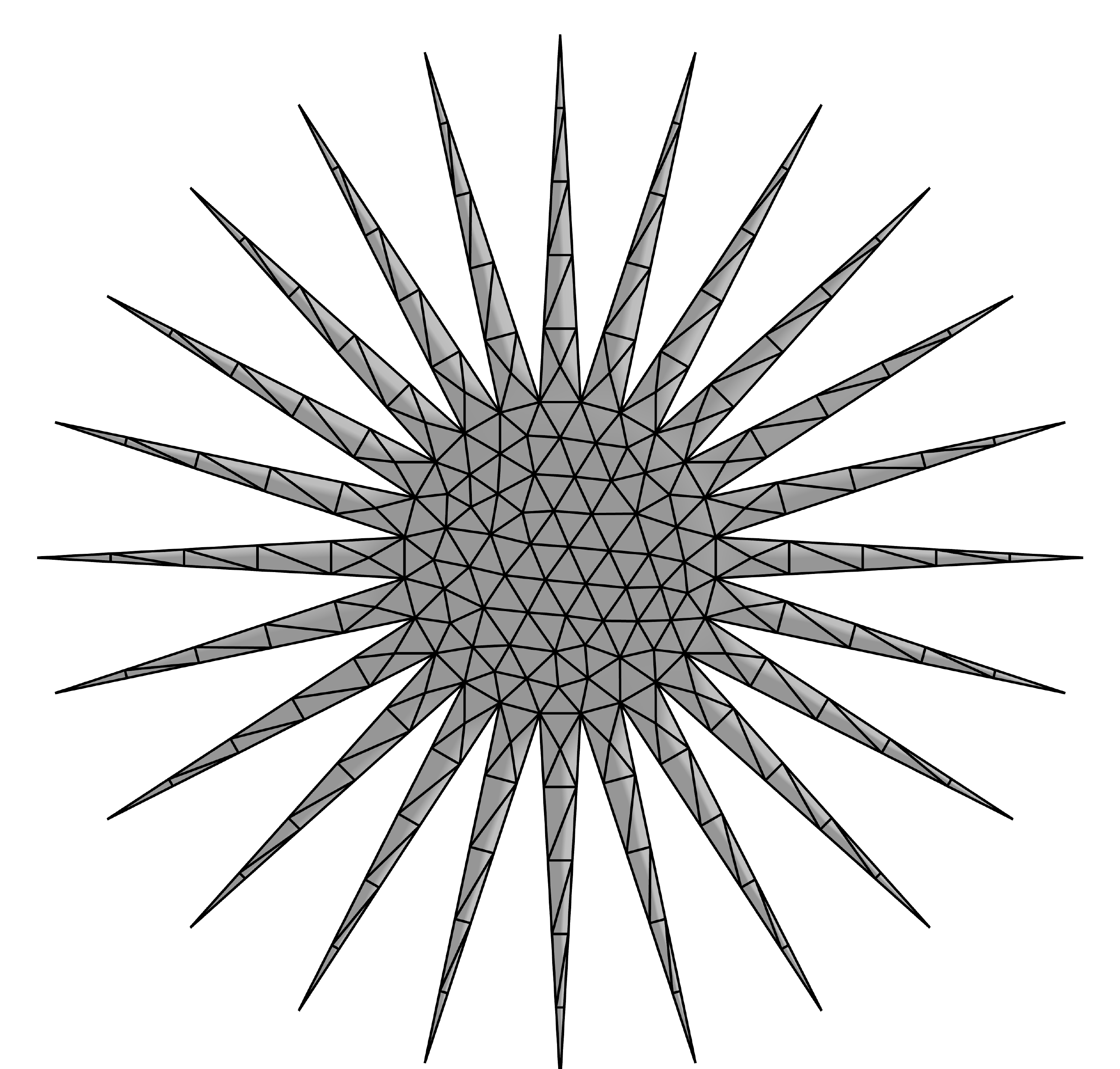}
        }
    \end{subfigure}
    \caption{Triangulations of the scatterers used in Examples \ref{subsec: a cuboid} and \ref{subsec: a star-based pyramid}. Left: a cuboid domain of dimensions $1.0\,\mathrm{m}\times 1.0\,\mathrm{m}\times 0.5\,\mathrm{m}$. Right: a star-based pyramid of height $0.5\,\mathrm{m}$ with a 24-pointed star-shaped base formed by alternating vertices on two circles of radii $1\,\mathrm{m}$ and $0.3\,\mathrm{m}$.}
    \label{fig: meshes}
\end{figure}

\subsection{A cuboid}
\label{subsec: a cuboid}

In this example, we choose an $x_2$-polarized plane wave
$$
\EEs^{\text{inc}} \paren{\xb} = \mathbf{e}_{2}\operatorname{exp}\paren{i\kappa \frac{x_1 + x_3}{\sqrt{2}}} \qqq \forall \xb = \paren{x_1, x_2, x_3}^\top\in \R^3,
$$
where $\mathbf{e}_{2}$ denotes the unit vector in the $x_2$-direction. We consider a cuboid scatterer $\Om_{\text{s}}$ of dimensions $1.0\,\mathrm{m}\times 1.0\,\mathrm{m}\times 0.5\,\mathrm{m}$, centered at $\paren{0.5, 0.5, 0.25}^\top$; see \Cref{fig: meshes} (left). Denoting by $L>0$ the characteristic length of $\Om_{\text{s}}$, we describe the material coefficients by
$$
\mu_{\text{s}}\paren{\xb} = \mu_0, \qqq \varepsilon_{\text{s}}\paren{\xb} = \varepsilon_0\paren{2 + \frac{81 x_1 x_2 x_3\norm{\xb}_2}{4 L^4}}\qqq \forall \xb = \paren{x_1, x_2, x_3}^\top \in \Om_{\text{s}}.
$$

\Cref{tab: a cuboid} reports the condition numbers of the three matrices, together with the corresponding GMRES iteration counts over a sequence of mesh refinements. The comparison between $\mathbb{B}_{\text{C}}$ and $\mathbb{B}_{\text{R}}$ shows that the reformulation yields a modest reduction in the condition numbers and lower GMRES iteration counts. However, this improvement is not sufficient to prevent $\mathbb{B}_{\text{R}}$ from becoming increasingly ill-conditioned as the mesh is refined. In contrast, the condition numbers of $\mathbb{PB}_{\mathrm{R}}$ remain stable across all considered mesh levels, in agreement with \Cref{theo: operator precondition}. A similar improvement is observed in the GMRES iteration counts: compared with the unpreconditioned system, the preconditioned one requires fewer iterations, and these counts vary only mildly as the mesh size decreases. Although no theoretical result is available in the present setting that directly relates the spectral condition number to the GMRES iteration counts, \Cref{tab: a cuboid} suggests that, for this example, the condition number nevertheless provides an empirical indicator of GMRES performance. 

\begin{table}[htbp]
\centering
\caption{Condition numbers and GMRES iteration counts over a sequence of mesh refinements. Example \ref{subsec: a cuboid}: A cuboid.}
\label{tab: a cuboid}
\setlength{\tabcolsep}{4pt}
\renewcommand{\arraystretch}{1.05}
\begin{adjustbox}{max width=\textwidth}
\begin{tabular}{
c
S[table-format=5.0]
S[table-format=5.2]
S[table-format=4.0]
S[table-format=5.0]
S[table-format=5.2]
S[table-format=4.0]
S[table-format=2.2]
S[table-format=2.0]
}
\toprule
\multirow{2}{*}{Mesh size $h$}
& \multicolumn{3}{c}{Costabel coupling}
& \multicolumn{5}{c}{Reformulated coupling} \\
\cmidrule(lr){2-4} \cmidrule(lr){5-9}
& \multicolumn{1}{r}{dof}
& \multicolumn{1}{r}{$\operatorname{cond}\paren{\mathbb{B}_{\mathrm C}}$}
& \multicolumn{1}{r}{iter.}
& \multicolumn{1}{r}{dof}
& \multicolumn{1}{r}{$\operatorname{cond}\paren{\mathbb{B}_{\mathrm R}}$}
& \multicolumn{1}{r}{iter.}
& \multicolumn{1}{r}{$\operatorname{cond}\paren{\mathbb{P}\mathbb{B}_{\mathrm R}}$}
& \multicolumn{1}{r}{iter.} \\
\midrule
$2.580\times 10^{-1}$ &   684 &  2979.84 &  500 &   528 &  2957.05 &  447 & 90.85 & 83 \\
$2.150\times 10^{-1}$ &  1139 &  4433.96 &  747 &   858 &  4406.05 &  665 & 91.34 & 90 \\
$1.792\times 10^{-1}$ &  1592 &  5184.26 &  938 &  1140 &  5140.97 &  818 & 91.84 & 90 \\
$1.493\times 10^{-1}$ &  2282 &  8428.86 & 1183 &  1524 &  8355.20 & 1010 & 91.73 & 90 \\
$1.244\times 10^{-1}$ &  4145 & 13281.13 & 1813 &  2520 & 13164.86 & 1513 & 92.56 & 93 \\
$1.037\times 10^{-1}$ &  5251 & 14304.38 & 2096 &  3006 & 14176.35 & 1718 & 92.10 & 93 \\
$8.640\times 10^{-2}$ &  8016 & 21008.50 & 2731 &  4182 & 20870.16 & 2238 & 92.34 & 94 \\
$7.200\times 10^{-2}$ & 11706 & 27890.56 & 3485 &  5550 & 27705.30 & 2746 & 92.43 & 94 \\
$6.000\times 10^{-2}$ & 20233 & 45333.37 & 5003 &  8466 & 44954.09 & 3655 & 92.47 & 96 \\
$5.000\times 10^{-2}$ & 29850 & 57809.84 & 6505 & 11484 & 57364.20 & 4386 & 93.08 & 97 \\
\bottomrule
\end{tabular}
\end{adjustbox}
\end{table}

\begin{figure}[!htp]
    \centering
    \includegraphics[width=\textwidth]{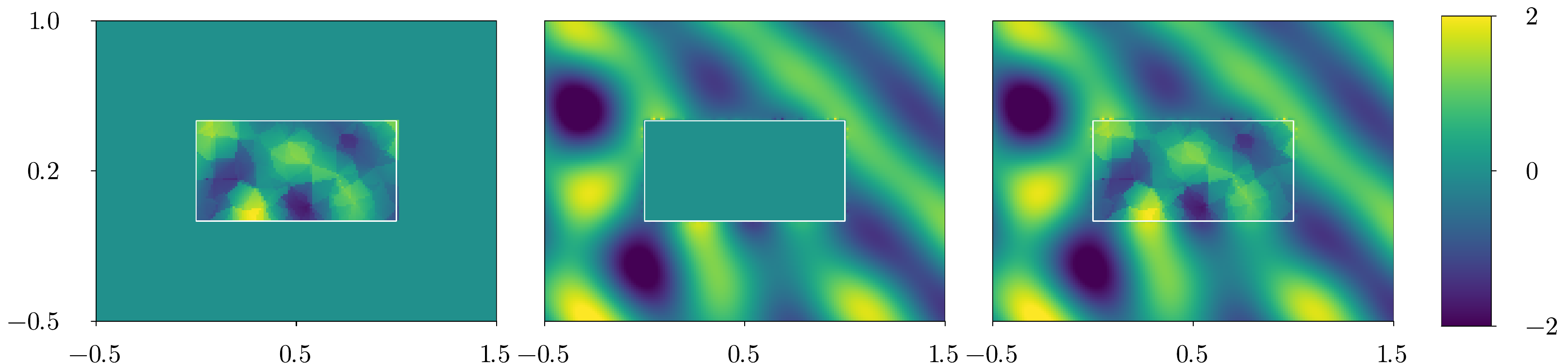}
    \caption{Real part of the $x_2$-component of the discrete total electric field over a selected rectangle in the plane $x_2 = 0.5$. Left: $\Om_{\text{s}}$, middle: $\Om_0$, right: $\Om_{\text{s}}\cup \Om_0$. The interface $\Gm\cap \left\{x_2=0.5\right\}$ is highlighted in white. Example \ref{subsec: a cuboid}: A cuboid, with mesh size $h=0.08\, \mathrm{m}$.}
    \label{fig: near field cuboid}
\end{figure}

As noted in \Cref{rem: advantage of the method}, once Problem \eqref{eq: discrete new formulation} has been solved, the solution $\paren{\mathcal{E}_h\, \bm{\tau}_h, \bm{\uplambda}_h}$ of Problem \eqref{eq: discrete old formulation} can be reconstructed. Here, $\EEs_h = \mathcal{E}_h\, \bm{\tau}_h$ is the discrete total electric field in $\Om_{\text{s}}$, whereas the discrete scattered and total electric fields in $\Om_0$ can be recovered from $\paren{\gm_D^{\text{s}} \EEs_h - \gm_D^0 \EEs^{\text{inc}}, \bm{\uplambda}_h}$ via the Stratton--Chu formula \cite[Lem.~2.4]{Mitrea} or \cite[Thm.~3]{BCS2002B}. For completeness, \Cref{fig: near field cuboid} shows the real part of the $x_2$-component of the discrete total electric field in a selected observation region. The observed continuity across $\Gm\cap \left\{x_2=0.5\right\}$ is consistent with the condition $\gm_D^{\text{s}}\EEs = \gm_D^0 \EEs + \gm_D^0 \EEs^{\text{inc}}$, since the displayed component is taken in the $\mathbf{e}_2$-direction, which is tangential to $\Gm$ along $\Gm \cap \left\{x_2=0.5\right\}$.

\subsection{A star-based pyramid}
\label{subsec: a star-based pyramid}

We now consider an $x_1$-polarized plane wave
$$
\EEs^{\text{inc}} \paren{\xb} = \mathbf{e}_{1}\operatorname{exp}\paren{i\kappa x_3} \qqq \forall \xb = \paren{x_1, x_2, x_3}^\top\in \R^3,
$$
where $\mathbf{e}_{1}$ denotes the unit vector in the $x_1$-direction. We adopt the geometry of $\Om_{\text{s}}$ in \cite[Sect.~7.2]{LC2024}, namely, a pyramid of height $0.5\,\mathrm{m}$ whose base is a $24$-pointed star, formed by alternating vertices on two circles of radii $1\,\mathrm{m}$ and $0.3\,\mathrm{m}$; see \Cref{fig: meshes} (right). This geometry is of particular interest because it defines a nonconvex polyhedron with sharp reentrant corners and reentrant edges, and thus provides a geometrically challenging test case. Finally, we set $\mu_{\text{s}}\paren{\xb} = 0.25\mu_0$ and $\varepsilon_{\text{s}}\paren{\xb} = 4\varepsilon_0$ for all $\xb\in \Om_{\text{s}}$.

\begin{table}[htbp]
\centering
\caption{Condition numbers and GMRES iteration counts over a sequence of mesh refinements. Example \ref{subsec: a star-based pyramid}: A star-based pyramid.}
\label{tab: a star-based pyramid}
\setlength{\tabcolsep}{4pt}
\renewcommand{\arraystretch}{1.05}
\begin{adjustbox}{max width=\textwidth}
\begin{tabular}{
c
S[table-format=5.0]
S[table-format=7.2]
S[table-format=5.0]
S[table-format=5.0]
S[table-format=6.2]
S[table-format=5.0]
S[table-format=3.2]
S[table-format=3.0]
}
\toprule
\multirow{2}{*}{Mesh size $h$}
& \multicolumn{3}{c}{Costabel coupling}
& \multicolumn{5}{c}{Reformulated coupling} \\
\cmidrule(lr){2-4} \cmidrule(lr){5-9}
& \multicolumn{1}{r}{dof}
& \multicolumn{1}{r}{$\operatorname{cond}\paren{\mathbb{B}_{\mathrm C}}$}
& \multicolumn{1}{r}{iter.}
& \multicolumn{1}{r}{dof}
& \multicolumn{1}{r}{$\operatorname{cond}\paren{\mathbb{B}_{\mathrm R}}$}
& \multicolumn{1}{r}{iter.}
& \multicolumn{1}{r}{$\operatorname{cond}\paren{\mathbb{P}\mathbb{B}_{\mathrm R}}$}
& \multicolumn{1}{r}{iter.} \\
\midrule
$1.600\times 10^{-1}$ & 5569  & 834753.46  & 4506  & 4608  & 131091.69 & 3649  & 216.44 & 236 \\
$1.393\times 10^{-1}$ & 8479  & 1259523.27 & 6544  & 6918  & 195516.07 & 5221  & 215.07 & 240 \\
$1.213\times 10^{-1}$ & 9085  & 1271603.35 & 6953  & 7350  & 216964.18 & 5493  & 216.95 & 239 \\
$1.056\times 10^{-1}$ & 11444 & 2122934.37 & 8584  & 9216  & 288468.26 & 6758  & 209.94 & 247 \\
$9.190\times 10^{-2}$ & 14381 & 2348581.35 & 10559 & 11502 & 397579.01 & 8254  & 205.47 & 255 \\
$8.000\times 10^{-2}$ & 17650 & 3434691.26 & 12623 & 13950 & 487766.17 & 9911  & 215.33 & 289 \\
$6.964\times 10^{-2}$ & 23991 & 5016198.24 & 16196 & 18726 & 728486.86 & 13037 & 189.55 & 274 \\
$6.063\times 10^{-2}$ & 28825 & 6545189.50 & 18303 & 22284 & 929057.78 & 15149 & 192.11 & 293 \\
\bottomrule
\end{tabular}
\end{adjustbox}
\end{table}

The results are summarized in \Cref{tab: a star-based pyramid}. For this star-shaped pyramid, the effect of the preconditioner is particularly pronounced. Compared with $\mathbb{B}_{\text{C}}$, the matrix $\mathbb{B}_{\text{R}}$ has a lower condition number at each mesh level. Moreover, the corresponding system is solved with fewer GMRES iterations. However, under mesh refinement, the condition number and GMRES iteration count still exhibit rapid growth. The behavior of the preconditioned matrix $\mathbb{PB}_{\mathrm{R}}$ is significantly different. Its condition number remains within a relatively narrow range and appears to be stable, in agreement with \Cref{theo: operator precondition}. The GMRES iteration counts for the preconditioned system do increase under mesh refinement, but much more slowly than those for the unpreconditioned system, and remain far smaller throughout. Once again, the condition number appears to correlate well with the GMRES iteration counts.

\section{Conclusions}
\label{sec: conclusions}

In this work, we derived an equivalent reformulation of the Costabel FEM-BEM coupling for the time-harmonic Maxwell transmission problem by means of a trace lifting operator and the Steklov--Poincar{\' e} operator. Assuming that the angular frequency $\om$ lies outside the resonance sets, we established the equivalence between the reformulated coupling and the Costabel coupling at both the continuous and discrete levels. Within the operator preconditioning framework, we then constructed a block-diagonal preconditioner for the discrete system arising from the reformulated coupling and proved that the spectral condition number of the preconditioned matrix is uniformly bounded.

A natural continuation of the present work is to derive a regularized formulation of Problem \eqref{eq: new formulation}, following the ideas of \cite{ABL2007, LC2024B}, thereby removing the assumption that $\om$ lies outside the resonance sets. Another direction is to extend the proposed preconditioning technique to other FEM-BEM coupling formulations.

\section*{Acknowledgment}
This work was supported by the European Research Council (ERC) under the European Union's Horizon 2020 Research and Innovation Program under Grant 101001847.


\bibliographystyle{elsarticle-num} 
\bibliography{abrv_ref}
\end{document}

%% file: lvc_style.tex
\newcommand{\R}{\mathbb R}
\newcommand{\C}{\mathbb C}

\DeclareMathOperator{\LLs}{\mathbf L}

\DeclareMathOperator{\Hs}{H}
\DeclareMathOperator{\HHs}{\mathbf H}

\DeclareMathOperator{\WWs}{\mathbf W}

\DeclareMathOperator{\Cs}{C}
\DeclareMathOperator{\CCs}{\mathbf C}

\DeclareMathOperator{\EEs}{\mathbf E}

\DeclareMathOperator{\UUs}{\mathbf U}

\DeclareMathOperator{\XXs}{\mathbf X}
\DeclareMathOperator{\YYs}{\mathbf Y}

\DeclareMathOperator{\ZZs}{\mathbf Z}

\newcommand{\paren}[1]{\left({#1}\right)}
\newcommand{\norm}[1]{\left\lVert{#1}\right\rVert}

\newcommand{\abs}[1]{\left\vert{#1}\right\vert}

\newcommand{\inprod}[2]{\left\langle{#1}, {#2}\right\rangle}

\newcommand{\vertiii}[1]{{\left\vert\kern-0.25ex\left\vert\kern-0.25ex\left\vert #1 
    \right\vert\kern-0.25ex\right\vert\kern-0.25ex\right\vert}}
    
\newcommand{\bm}[1]{\boldsymbol{#1}}

\newcommand{\curlt}{\textbf{\textup{curl}}}

\newcommand{\gm}{\gamma}
\newcommand{\Gm}{\Gamma}
\newcommand{\om}{\omega}
\newcommand{\Om}{\Omega}

\newcommand{\vphi}{\varphi}

\newcommand{\xb}{\bm{x}}
\newcommand{\yb}{\bm{y}}

\newcommand{\vb}{\bm{v}}
\newcommand{\ub}{\bm{u}}

\newcommand{\tv}{\textbf{t}}

\newcommand{\dx}{\,\mathrm{d}\xb}

\newcommand{\ds}{\,\mathrm{d}s}

\newcommand{\q}{\quad}
\newcommand{\qq}{\qquad}
\newcommand{\qqq}{\qquad\quad}

%% file: thm_style.tex
\newtheorem{theorem}{Theorem}[section]
\newtheorem{lemma}{Lemma}[section]
\newtheorem{proposition}[theorem]{Proposition}

\newtheorem{remark}{Remark}[section]
\newtheorem{assumption}{Assumption}